\documentclass[12pt,reqno]{amsart}
\usepackage[a4paper]{geometry}
\usepackage{graphicx} 
\usepackage{amssymb,amsmath,amsthm,enumerate}
\usepackage{array} 

\DeclareMathOperator{\Dom}{Dom}

\DeclareMathOperator{\AC}{AC}

\renewcommand{\Re}{\operatorname{Re}}
\renewcommand{\Im}{\operatorname{Im}}

\newcommand{\abs}[1]{\lvert#1\rvert}

\newcommand{\norm}[1]{\lVert#1\rVert}
\newcommand{\Norm}[1]{\left\lVert#1\right\rVert}
\newcommand{\jap}[1]{\langle#1\rangle}

\newcommand{\bbR}{{\mathbb R}}
\newcommand{\bbC}{{\mathbb C}}

\newcommand{\bH}{\mathbf{H}}
\newcommand{\bR}{\mathbf{R}}

\newcommand{\calK}{\mathcal{K}}
\newcommand{\calO}{\mathcal{O}}

\newcommand{\dd}{\mathrm d}
\newcommand{\ii}{\mathrm i}
\newcommand{\ee}{\mathrm e}

\newcommand{\myF}{{F}}
\newcommand{\oldF}{{\mathbf{F}}}
\newcommand{\newe}{{\mathbf{e}}}

\numberwithin{equation}{section}

\theoremstyle{plain}
\newtheorem{theorem}{\bf Theorem}[section]
\newtheorem*{theorem*}{Theorem}
\newtheorem{lemma}[theorem]{\bf Lemma}
\newtheorem{proposition}[theorem]{\bf Proposition}
\newtheorem*{proposition*}{\bf Proposition}

\theoremstyle{definition}
\newtheorem{definition}[theorem]{\bf Definition}
\newtheorem*{definition*}{\bf Definition}

\theoremstyle{remark}
\newtheorem*{remark*}{\bf Remark}
\newtheorem{remark}[theorem]{\bf Remark}

\newcommand{\eps}{\varepsilon}

\newcommand{\alignx}[1]{\makebox[2.2cm][r]{$#1$}\:{\!$x$}} 

\newcommand\smallO{ 
  \mathchoice
    {{\scriptstyle\mathcal{O}}}
    {{\scriptstyle\mathcal{O}}}
    {{\scriptscriptstyle\mathcal{O}}}
    {\scalebox{.7}{$\scriptscriptstyle\mathcal{O}$}}
  }

\begin{document}

\title[Schr\"odinger operators with integrable complex potentials]{Schr\"odinger operators on the half-line with integrable complex potentials}

\author{Alexander Pushnitski}
\address{Department of Mathematics, King's College London, Strand, London, WC2R~2LS, U.K.}
\email{alexander.pushnitski@kcl.ac.uk}

\author{Franti\v{s}ek \v{S}tampach}
\address{Department of Mathematics, Faculty of Nuclear Sciences and Physical Engineering, Czech Technical University in Prague, Trojanova 13, 12000 Prague~2, Czech Republic.}
\email{frantisek.stampach@cvut.cz}

\dedicatory{To Barry Simon on his 80th birthday, with admiration for his work}

\date{\today}

\begin{abstract}
In our previous work, we introduced the concept of a \emph{spectral pair} for a half-line Schrödinger operator with a \emph{complex} bounded potential $q$, serving as a substitute for the spectral measure in a non-self-adjoint setting. In this paper, we study the case of $q \in L^1(\mathbb{R}_+)$. We derive explicit formulas for the spectral pair in terms of the Jost solutions of a system of two equations naturally associated with the non-self-adjoint Schrödinger operator. A key component of our work, which is of independent interest, is the existence proof and analysis of these Jost solutions.
\end{abstract}

\maketitle

\setcounter{tocdepth}{1}
\tableofcontents

\section{Introduction} 

\subsection{The regular solutions and the $m$-function}
We start by recalling the definition of the $m$-function of the half-line Schr\"odinger operator in the notation adapted to our purposes. We denote $\bbR_+=(0,\infty)$. Let $q$ be a \emph{real-valued} locally integrable  function on $\bbR_+$, and let $H$ be the self-adjoint Schr{\"o}dinger operator
\begin{equation}
Hf=-f''+qf
\label{eq:def_H}
\end{equation}
on $L^2(\bbR_+)$ with the boundary condition 
\begin{equation}
f'(0)+\alpha f(0)=0
\label{bc0}
\end{equation}
at the origin, where $\alpha\in\mathbb{R}\cup\{\infty\}$. The case $\alpha=\infty$ corresponds to the Dirichlet boundary condition $f(0)=0$. We assume further that $q$ is limit-point at infinity, see e.g. \cite[Section X.1]{RS2}. 

In order to make the formulas below more readable, we introduce the \emph{boundary functionals}
\begin{align}
\ell_\alpha(f)&=\frac{\alpha f'(0)-f(0)}{\sqrt{1+\alpha^2}},
\quad
\ell_\alpha^\perp(f)=\frac{f'(0)+\alpha f(0)}{\sqrt{1+\alpha^2}},
\quad &&\text{ if $\alpha\not=\infty$},
\label{eq:sa6}
\\
\ell_\alpha(f)&=f'(0), \hskip2.05cm 
\ell_\alpha^\perp(f)=f(0),
&&\text{ if $\alpha=\infty$}.
\label{eq:sa7}
\end{align}
With this notation, the boundary condition \eqref{bc0} can be written as $\ell_\alpha^\perp(f)=0$.

Let $\varphi,\theta$ be the solutions of the eigenvalue equation 
\begin{equation}
-f''(x)+q(x)f(x)=\lambda f(x), \quad x>0,
\label{eq:eveq}
\end{equation}
satisfying the boundary conditions
\[
\left\{
\begin{aligned}
\ell_\alpha(\varphi)&=-1,\\
\ell_\alpha^\perp(\varphi)&=0,
\end{aligned}\right.
\qquad
\left\{
\begin{aligned}
\ell_\alpha(\theta)&=0,\\
\ell_\alpha^\perp(\theta)&=1.
\end{aligned}\right.
\]
Recall that, for every non-real $\lambda$, the Titchmarsh--Weyl $m$-function is defined as the unique complex number $m_{\alpha}(\lambda)$ such that the solution
\begin{equation}
\chi=\theta-\varphi m_{\alpha}(\lambda)\in L^2(\bbR_+).
\label{X7}
\end{equation}

\subsection{The spectral measure and the Borg--Marchenko uniqueness theorem}

The $m$-function is a Herglotz--Nevanlinna function in the upper half-plane, i.e., it is analytic in the half-plane $\Im \lambda>0$ and satisfies $\Im m_\alpha (\lambda)>0$ therein. As a consequence, it can be represented as
\begin{equation}
m_\alpha(\lambda)=
\Re m_\alpha(\ii)+\int_{-\infty}^\infty \left(\frac{1}{t-\lambda}-\frac{t}{1+t^2}\right)\dd\sigma(t), 
\quad \Im \lambda>0,
\label{eq.r4}
\end{equation}
where $\sigma$ is a unique positive measure called the \emph{spectral measure} of $H$.  The term ``spectral measure'' is justified by the fact that $H$ is unitarily equivalent to the operator of multiplication by the independent variable in $L^2_\sigma(\bbR)$, i.e. the space of square-integrable functions on $\bbR$  with respect to the measure $\sigma$. In particular, the support of $\sigma$ coincides with the spectrum of $H$, and the point masses of $\sigma$ coincide with the eigenvalues of $H$. 

The classical Borg--Marchenko uniqueness theorem \cite{Borg,M} says that the \emph{spectral map} 
\[
(q,\alpha)\mapsto \sigma
\]
is injective, i.e. the spectral measure $\sigma$ uniquely determines both the potential $q$ and the boundary parameter $\alpha$. More recent works with alternative proofs and further developments include~\cite{GS1,GS2,S1} and~\cite{B}.

\subsection{Integrable potentials}
Now let us assume, in addition, that $q\in L^1(\bbR_+)$. Then much more can be said about the spectral measure $\sigma$: it is absolutely continuous on $[0,\infty)$ and pure point on $(-\infty,0)$ with $0$ being the only possible point of accumulation of point masses. Moreover, $\sigma$ can be expressed in terms of the \emph{Jost function}. Below we recall these well-known results, which are usually discussed in the framework of scattering theory of Schr{\" o}dinger operators on $\bbR_{+}$, see e.g. \cite[\S~XI.8]{RS3} and \cite[Chap.~4]{Yaf}.

For any $k\in\bbC\setminus\{0\}$ with $\Im k\geq0$ there exists a \emph{Jost solution} which we will denote by $f_{+\ii}(x,k)$; this is the  unique solution of the eigenvalue equation \eqref{eq:eveq} with $\lambda=k^2$ satisfying the asymptotics
\begin{equation}
f_{+\ii}(x,k)=\ee^{\ii kx}+\smallO(\ee^{\ii kx}), \quad x\to\infty
\label{eq:sa4}
\end{equation}
(the subscript $+\ii$ refers to the factor in front of $kx$ in the exponential). 
The Wronskian
\begin{equation}
f_{+\ii}'(\cdot,k)\varphi(\cdot,k^2)-f_{+\ii}(\cdot,k)\varphi'(\cdot,k^2)=\ell_\alpha^\perp(f_{+\ii}(\cdot,k))
\label{eq:sa4a}
\end{equation}
is known as the \emph{Jost function}. 
With this notation, the spectral measure $\sigma$ can be described as follows. 

\emph{On the positive half-line:} the measure $\sigma$ is absolutely continuous with respect to the Lebesgue measure, with the density
\begin{equation}
\frac{\dd\sigma}{\dd\lambda}(\lambda)=\frac{k}{\pi\abs{\ell_\alpha^\perp(f_{+\ii}(\cdot,k))}^2}, 
\quad \text{ where }
k=\sqrt{\lambda}>0.
\label{eq:sa2}
\end{equation}
The Jost function in the denominator does not vanish for $k>0$. 

\emph{On the negative half-line:}
for $\lambda<0$ let us denote $\kappa=\sqrt{\abs{\lambda}}$ and $f_{-1}(x,\kappa)=f_{+\ii}(x,\ii\kappa)$. Observe that \eqref{eq:sa4} can be rewritten as 
\begin{equation}
f_{-1}(x,\kappa)=\ee^{-\kappa x}+\smallO(\ee^{-\kappa x}), \quad x\to\infty.
\label{eq:sa5}
\end{equation}
One finds that $\lambda<0$ is an eigenvalue of $H$ if and only if $\ell_\alpha^\perp(f_{-1}(\cdot,\kappa))=0$ and in this case we have
\begin{equation}
\sigma(\{\lambda\})=\frac{\abs{\ell_\alpha(f_{-1}(\cdot,\kappa))}^2}{\norm{f_{-1}(\cdot,\kappa)}_{L^2}^2},
\quad 
\text{ where }
\kappa=\sqrt{\abs{\lambda}}>0.
\label{eq:sa1}
\end{equation}
\begin{remark*}
In fact, formula \eqref{eq:sa1} does not really require the potential $q$ to be integrable. It holds for any eigenvalue of $H$ with $f_{-1}$ replaced by any $L^2$-solution of the eigenvalue equation \eqref{eq:eveq}.
\end{remark*}

\subsection{Non-self-adjoint case: overview}
The main object of this paper is the Schr{\"o}dinger operator $H$ on $L^2(\bbR_+)$ with a \emph{complex-valued} potential $q$ and a \emph{non-self-adjoint} boundary condition 
\begin{equation}
f'(0)+\alpha f(0)=0,
\label{eq:bc0}
\end{equation}
where $\alpha\in\bbC\cup\{\infty\}$. 
In order to simplify this preliminary discussion, let us first assume that the potential $q$ is \emph{bounded} (we will lift this assumption later). Then the Schr{\"o}dinger operator $H$ can be defined by formula \eqref{eq:def_H} on the domain 
\[
\Dom H=\{f\in W^{2,2}([0,\infty)): \text{$f$ satisfies \eqref{eq:bc0}} \}.
\]
With this definition, $H$ is densely defined and closed in $L^2(\bbR_+)$. Since both $q$ and $\alpha$ can be complex-valued, the operator $H$ is in general non-self-adjoint.

In  \cite{PS3}, the present authors introduced the concept of a \emph{spectral pair} for $H$ (see Definition~\ref{def:sp-pair} below), which should be viewed as a substitute for the spectral measure $\sigma$. Our spectral pair is $(\nu,\psi)$, where $\nu$ is a measure on the positive half-line and $\psi$ is a complex-valued function on the support of $\nu$ such that $\abs{\psi(\lambda)}\leq 1$ for $\nu$-a.e. $\lambda>0$. In particular, if $H$ is self-adjoint and positive semi-definite, then $\psi\equiv1$ and $\nu$ reduces to the standard spectral measure $\sigma$ of $H$ (up to the normalisation factor of $1/2$).

In \cite{PS3}, a Borg--Marchenko-type uniqueness theorem was established, asserting that the map 
\[
(q,\alpha)\mapsto (\nu,\psi)
\]
is an injection, i.e. both $q$ and $\alpha$ are uniquely determined by the spectral pair~$(\nu,\psi)$.

\medskip

\emph{
The main aim of this paper is to study the operator $H$ under the condition $q\in L^1(\bbR_+)$, with both $q$ and $\alpha$ complex-valued, and to derive analogues of formulas \eqref{eq:sa2} and \eqref{eq:sa1}  for the spectral pair $(\nu,\psi)$.
}

\medskip

Before we proceed, we need to address two minor technical points associated with unbounded $q$. The first one is that the definition of the operator $H$ under the constraint $q\in L^1(\bbR_+)$, without assuming that $q$ is bounded, is not completely trivial. In Appendix~\ref{sec:B}, we explain that $H$ is a closed operator with the domain 
\begin{align}
\Dom H=\{f\in L^2(\bbR_+): &\;\text{$f$ and $f'$ are absolutely continuous on $\bbR_+$, }
\notag
\\
&\text{$-f''+qf\in L^2(\bbR_+)$ and \eqref{eq:bc0} holds}\}\ .
\label{eq:domainH}
\end{align}
The fact that this domain is dense in $L^2(\bbR_+)$ is not obvious (because it does not necessarily contain all smooth compactly supported functions) but true, see Appendix~\ref{sec:B}.

The second technical point is that, strictly speaking, the framework of \cite{PS3} is not applicable to $q\in L^1(\bbR_+)$ because in \cite{PS3} it was assumed that $q$ is bounded. In Appendix~\ref{sec:C}, we explain how to extend the construction of \cite{PS3} to the case of unbounded $q\in L^1(\bbR_+)$; there are no significant difficulties here. 

With these technical points out of the way, we can informally explain the nature of our main results. Similarly to the classical self-adjoint theory, condition $q\in L^1(\bbR_+)$ ensures the existence of Jost-type solutions of the eigenvalue equation for $H$ with prescribed asymptotics at infinity, analogous to \eqref{eq:sa4} and \eqref{eq:sa5}. More precisely, it turns out that instead of a single eigenvalue equation \eqref{eq:eveq}, we are forced to consider a \emph{system} of equations
\begin{equation}
\left\{
\begin{aligned}
-f_1''+\overline{q}f_1&=k^2f_2,
\\
-f_2''+qf_2&=k^2f_1.
\end{aligned}\right.
\label{eq:ev-eq2}
\end{equation}
Note the complex conjugation over $q$ in the first equation! We prove that for any non-zero $k\in\bbC$, any solution $f_1$, $f_2$ to this system has asymptotics involving exponentials $\ee^{\pm kx}$ and $\ee^{\pm\ii kx}$ as $x\to\infty$. Furthermore, we identify solutions with prescribed asymptotics at infinity, see Theorem~\ref{thm.e1} below. By a slight abuse of terminology, we will call them \emph{Jost solutions}. The proof of Theorem~\ref{thm.e1}  (in Section~\ref{sec.e}) turns out to be more subtle than in the classical self-adjoint case because of the coexistence of the ``hyperbolic'' $\ee^{\pm kx}$ and ``elliptic'' $\ee^{\pm\ii kx}$ modes for the same spectral parameter $k^2$. At the technical level, this is reflected in the fact that the corresponding integral equations are in general of Fredholm type but not of Volterra type. We regard the asymptotic analysis of the system \eqref{eq:ev-eq2} as our first main result, which is of independent interest. 

Our second main result is Theorem~\ref{thm.5.3}, where we express both $\nu$ and $\psi$ in terms of the Jost solutions in the spirit of \eqref{eq:sa2} and \eqref{eq:sa1}. 

In order to describe our third main result, we come back for a moment to the classical self-adjoint case. We recall that for $k>0$, the Jost solution $f_{+\ii}$ and its complex conjugate form a basis of the two-dimensional space of solutions of \eqref{eq:eveq}, and so one can expand the regular solution $\varphi$ as a linear combination of these, 
\[
\varphi(x,k^2)=\gamma_{+\ii}(k) f_{+\ii}(x,k)+\gamma_{-\ii}(k)\overline{f_{+\ii}(x,k)}.
\]
The coefficients $\gamma_{+\ii}(k)$ and $\gamma_{-\ii}(k)$ are easy to compute as 
\[
\overline{\gamma_{+\ii}(k)}=\gamma_{-\ii}(k)=\ell_\alpha^\perp(f_{+\ii})/(2\ii k).
\]
This gives a formula alternative to \eqref{eq:sa2},
\begin{equation}
\frac{\dd\sigma}{\dd\lambda}(\lambda)
=\frac1{4\pi k\abs{\gamma_{+\ii}(k)}^2},
\quad 
k=\sqrt{\lambda}>0,
\label{X9}
\end{equation}
for the spectral measure on the positive half-line.

Our third main result is Theorem~\ref{thm:a4}, where we give an alternative expression for $\nu$ and $\psi$ in the spirit of \eqref{X9}. Here our considerations have much in common with scattering theory. 

Lastly, in order to inform the reader's intuition on the spectral pair, we compute formulas for $\nu$ and $\psi$ in the Born approximation, i.e. the linearised forms of $\nu$ and $\psi$ for small $q$. 

\subsection{The structure of the paper}
In Section~\ref{sec:mr-a}, we explain our matrix formalism (that we call \emph{hermitisation framework}) for analysing the system \eqref{eq:ev-eq2} and define the spectral pair $(\nu,\psi)$. 
In Section~\ref{sec:mr-b}, we introduce the Jost solutions of \eqref{eq:ev-eq2} and state our main Theorems~\ref{thm.5.3} and \ref{thm:a4}. Jost solutions are studied in Sections~\ref{sec.e} and~\ref{sec:cont}. Theorems~\ref{thm.5.3} and \ref{thm:a4} are proved in Sections~\ref{sec.ee} and \ref{sec:wr}. In Section~\ref{sec:Born}, we compute $\nu$ and $\psi$ in the Born approximation. Appendices~\ref{sec:B} and~\ref{sec:C} with supplementary results are included at the end of the paper.

\subsection{Matrix and vector notation}
If $a$ and $b$ are two vectors in $\bbC^2$, we will denote by $\{a,b\}$ the $2\times2$ matrix with the columns $a$ and $b$, i.e.
\[
\{a,b\}:=\begin{pmatrix}a_{1}&b_{1}\\a_{2}&b_{2}\end{pmatrix}, 
\quad\text{where }
a=\begin{pmatrix}a_{1}\\a_{2}\end{pmatrix}
\text{ and }\;
b=\begin{pmatrix}b_{1}\\b_{2}\end{pmatrix}.
\]
Next, we denote 
\begin{equation}
e_1=\begin{pmatrix}1\\0\end{pmatrix}, 
\quad
e_2=\begin{pmatrix}0\\1\end{pmatrix}
\quad
\mbox{ and }
\quad
e_{+}=\begin{pmatrix}1\\ 1\end{pmatrix},
\quad
e_{-}=\begin{pmatrix}1\\-1\end{pmatrix}.
\label{eq:a4aa}
\end{equation}
Vectors $e_\pm$ are the eigenvectors of the Pauli matrix 
\[
 \epsilon = \begin{pmatrix} 0 & 1 \\ 1 & 0 \end{pmatrix},
\]
which enters the definition of the hermitisation of $H$ (see~\eqref{eq:a2} below) and appears frequently in the paper.
We have $\epsilon e_{\pm}=\pm e_+$. We denote by $P_+$ and $P_-$ the orthogonal projections onto the vectors $e_+$ and $e_-$ in $\bbC^2$, i.e.
\begin{equation}
P_{+}=\frac12
\begin{pmatrix}
1&1\\ 1&1
\end{pmatrix}
\quad\text{ and }\quad
P_{-}=\frac12
\begin{pmatrix}
1&-1\\ -1&1
\end{pmatrix}.
\label{eq:proj}
\end{equation}

\section{The hermitisation framework}\label{sec:mr-a}

\subsection{Definition of the Schr\"odinger operator $H$}
For complex-valued $q\in L^1(\bbR_+)$ and $\alpha\in\bbC\cup\{\infty\}$, we define the Schr\"odinger operator $H=H(q,\alpha)$ in $L^2(\bbR_+)$ by 
\[
Hf=-f''+qf
\]
on the domain \eqref{eq:domainH}. 

\begin{proposition}\label{prp:B1}
For any $q\in L^1(\bbR_+)$ and $\alpha\in\bbC\cup\{\infty\}$, the operator $H(q,\alpha)$ is densely defined and closed. 
The adjoint satisfies $H(q,\alpha)^*=H(\overline{q},\overline{\alpha})$. 
\end{proposition}
We indicate the proof of Proposition~\ref{prp:B1} and discuss the relevant literature in Appendix~\ref{sec:B}. 

\begin{remark*}
In principle, one can define $H$ with any complex $q\in L^1_{\mathrm{loc}}(\bbR_+)$. However, this necessitates the discussion of non-self-adjoint analogues of the limit-point/limit-circle dichotomy, which is quite delicate; we refer the interested reader to \cite{DG}. If $q\in L^1(\bbR_+)$, we are in the limit-point case, although this terminology is not standard in the non-self-adjoint setting.
\end{remark*}

In analogy with \eqref{eq:sa6} and \eqref{eq:sa7}, we define the boundary functionals in the non-self-adjoint case by 
\begin{align}
\ell_\alpha(f)&=\frac{\overline{\alpha} f'(0)-f(0)}{\sqrt{1+\abs{\alpha}^2}},
\quad
\ell_\alpha^\perp(f)=\frac{f'(0)+\alpha f(0)}{\sqrt{1+\abs{\alpha}^2}},
\quad &&\text{ if $\alpha\not=\infty$},
\label{eq:nsa6}
\\
\ell_\alpha(f)&=f'(0), \hskip2.05cm 
\ell_\alpha^\perp(f)=f(0),
&&\text{ if $\alpha=\infty$}.
\label{eq:nsa7}
\end{align}
The difference from \eqref{eq:sa6} is in the complex conjugation over $\alpha$ in the first formula of \eqref{eq:nsa6}. In this notation, the boundary condition \eqref{eq:bc0} becomes $\ell_\alpha^\perp(f)=0$. Below we will use a convenient expression \eqref{eq:wr4} for Wronskians in terms of these boundary operators.

\subsection{The hermitised operator $\bH$}
Our approach is to access the spectral properties of the non-self-adjoint operator $H=H(q,\alpha)$ by using the language of self-adjoint spectral theory. To this end, we consider the \emph{hermitisation} of $H$, i.e. the \emph{self-adjoint} block-matrix operator
\[
\bH=
\begin{pmatrix}
0& H\\ H^*& 0
\end{pmatrix}
\quad\text{ in } L^2(\bbR_+)\oplus L^2(\bbR_+)
\]
on the domain $\Dom H^*\oplus\Dom H$; the adjoint $H^*$ is described in Proposition~\ref{prp:B1}.

Identifying $L^2(\bbR_+)\oplus L^2(\bbR_+)$ with the space of $\bbC^2$-valued functions $L^2(\bbR_+;\bbC^2)$, we can view $\bH$ as the operator
\begin{equation}
\quad 
\bH=-\epsilon\frac{\dd^2}{\dd x^2}+Q 
\quad\mbox{ with }
\epsilon=
\begin{pmatrix}
0&1\\1&0
\end{pmatrix}
\text{ and }\,
Q=\begin{pmatrix}0&q\\ \overline{q}&0\end{pmatrix},
\label{eq:a2}
\end{equation}
acting in $L^2(\bbR_+;\bbC^2)$. 
The operator $\bH$ is supplied with the boundary condition 
\begin{equation}
F'(0)+AF(0)=0,
\quad\text{ where }
A=\begin{pmatrix}
\overline{\alpha} & 0 \\ 0 & \alpha
\end{pmatrix}.
\label{eq:a2a0}
\end{equation}
If $\alpha=\infty$, \eqref{eq:a2a0} is to be interpreted as the Dirichlet boundary condition $F(0)=0$. In analogy with \eqref{eq:nsa6} and \eqref{eq:nsa7}, it is convenient to define the boundary operators $L_\alpha$, $L_\alpha^\perp$ for $\bbC^2$-valued vector-functions by 
\begin{align}
L_\alpha(F)=
\begin{pmatrix}
\ell_{\overline{\alpha}}(F_1)\\ \ell_\alpha(F_2)
\end{pmatrix}
\quad\text{ and }\quad
L_\alpha^\perp(F)=
\begin{pmatrix}
\ell_{\overline{\alpha}}^\perp(F_1)\\ \ell_\alpha^\perp(F_2)
\end{pmatrix},
\quad\text{ where }
F=\begin{pmatrix}
F_1\\ F_2
\end{pmatrix}.
\label{eq:e30c}
\end{align}
Alternatively, we can write
\begin{align}
L_\alpha(F)&=\frac{\overline{A} F'(0)-F(0)}{\sqrt{1+\abs{\alpha}^2}},
\quad   
L_\alpha^\perp(F)=\frac{F'(0)+A F(0)}{\sqrt{1+\abs{\alpha}^2}},
&& \text{ if $\alpha\not=\infty$},
\label{eq:nsa6v}
\\
L_\alpha(F)&=F'(0), \hskip2.2cm 
L_\alpha^\perp(F)=F(0), 
&&\text{ if $\alpha=\infty$}.
\label{eq:nsa7v}
\end{align}
With this notation, the boundary condition \eqref{eq:a2a0} becomes $L_\alpha^\perp(F)=0$ and $\bH$ is a~self-adjoint operator in $L^2(\bbR_+;\bbC^2)$ with the domain 
\begin{align*}
\Dom\bH=\{F\in L^2(\bbR_+;\bbC^2): &\text{ $F$ and $F'$ are absolutely continuous on $\bbR_+$, }
\\
&-\epsilon F''+QF\in L^2(\bbR_+;\bbC^2) \text{ and } L_\alpha^\perp(F)=0\,\}.
\end{align*}
It is important to note that by a general operator theoretic argument, $\bH$ is unitarily equivalent to the operator
\begin{equation}
\begin{pmatrix}
\sqrt{H^*H}&0\\ 0&-\sqrt{HH^*}
\end{pmatrix}\, .
\label{eq:e30b}
\end{equation}
In particular, $-\bH$ is unitarily equivalent to $\bH$, so the spectrum of $\bH$ is symmetric with respect to the reflection $\lambda\mapsto-\lambda$.

\subsection{The Titchmarsh--Weyl $M$-function}
Let us denote by $\Phi,\Theta$ the $2\times2$ matrix-valued solutions (the fundamental system of regular solutions) of the eigenvalue equation 
\begin{equation}
-\epsilon F''(x,\lambda)+Q(x)F(x,\lambda)=\lambda F(x,\lambda), \quad x>0,
\label{eq:init_diff_eq}
\end{equation}
with the Cauchy data at $x=0$:
\begin{equation}
\left\{
\begin{aligned}
L_\alpha(\Phi)&=-I,
\\
L_\alpha^\perp(\Phi)&=0,
\end{aligned}
\right.
\qquad
\left\{
\begin{aligned}
L_\alpha(\Theta)&=0,
\\
L_\alpha^\perp(\Theta)&=\epsilon,
\end{aligned}
\right.
\label{eq:phi,theta_bc_cond}
\end{equation}
where $\lambda\in\bbC$. Here the boundary operators $L_\alpha$, $L_\alpha^\perp$ are applied to $2\times2$ matrices rather than vectors, and should be understood as \eqref{eq:nsa6v}, \eqref{eq:nsa7v}. Observe that $\Phi$ (but not $\Theta$) satisfies the boundary condition \eqref{eq:a2a0}. 

The following statement is the direct analogue of the definition~\eqref{X7} of the scalar $m$-function. See Appendix~\ref{sec:C} for a short discussion of the proof.

\begin{proposition}\label{prp.Mf}
Let $q\in L^1(\bbR_+)$ and let $Q$ be as in \eqref{eq:a2}. 
For every $\lambda\in\bbC\setminus\bbR$, there exists a unique $2\times 2$ matrix $M_{\alpha}(\lambda)$ such that both columns of the matrix
\begin{equation}
X(x,\lambda)=\Theta(x,\lambda)-\Phi(x,\lambda)M_\alpha(\lambda)
\label{eq:X}
\end{equation}
belong to $L^2(\bbR_+;\bbC^2)$ as functions of $x$. The matrix-valued function $M_\alpha$ is analytic in $\Im\lambda>0$ and satisfies
\begin{equation}
\Im M_\alpha(\lambda)>0 \; \text{ if }\Im \lambda>0.
\label{eq:a4a}
\end{equation}
\end{proposition}
Here and below, we use the notation
\[
 \Re M = \frac{1}{2}(M+M^{*}) \quad\mbox{ and }\quad \Im M = \frac{1}{2\ii}(M-M^{*})
\]
for a square matrix $M$, and $M>0$ means that $M$ is positive definite.

\subsection{The spectral measure $\Sigma$ and the spectral pair $(\nu,\psi)$}

Proposition~\ref{prp.Mf} implies that $M_{\alpha}$ is a matrix-valued Herglotz--Nevanlinna function and therefore, by the general integral representation theorem, see~\cite[Theorems~2.3 and 5.4]{GT}, we have (cf. \eqref{eq.r4})
\begin{equation}
M_\alpha(\lambda)=
\Re M_\alpha(\ii)+\int_{-\infty}^\infty \left(\frac{1}{t-\lambda}-\frac{t}{1+t^2}\right)\dd\Sigma(t), \quad 
\Im\lambda>0,
\label{eq:intres}
\end{equation}
where $\Sigma$ is a unique $2\times 2$ matrix-valued measure on $\bbR$, called the \emph{spectral measure} of $\bH$.
\begin{proposition}\label{prp.sm}
Let $q\in L^1(\bbR_+)$ and let $Q$ be as in \eqref{eq:a2}. 
\begin{enumerate}[\rm (i)]
\item
The operator $\bH$ is unitarily equivalent to the operator of multiplication by the independent variable in the Hilbert space $L_\Sigma^2(\bbR;\bbC^2)$, i.e. the space of $\bbC^2$-valued functions square-integrable with respect to the measure $\Sigma$.
\item
The measure $\Sigma$ has the structure
\begin{equation}
\dd\Sigma=
\begin{pmatrix}
1&\psi
\\
\overline{\psi}&1
\end{pmatrix}\dd\nu,
\label{eq:a1}
\end{equation}
where $\nu$ is an \emph{even} measure on $\bbR$ and $\psi\in L^\infty(\nu)$ is an \emph{odd} complex-valued function satisfying $\abs{\psi(\lambda)}\leq 1$ for $\nu$-a.e. $\lambda\in\bbR$.
\item
The spectral map 
\[
L^1(\bbR_+)\times(\bbC\cup\{\infty\})\ni
(q,\alpha)\mapsto (\nu,\psi)
\]
is an injection, i.e. the pair $(\nu,\psi)$ uniquely determines both $q$ and $\alpha$. 
\end{enumerate}
\end{proposition}
Part (i) can be regarded as known and parts (ii) and (iii) are due to \cite{PS3}. We comment on the proof in Appendix~\ref{sec:C}.

 Since $\nu$ is odd and $\psi$ is even, the pair $(\nu,\psi)$ is uniquely determined by its restriction onto $[0,\infty)$. 
\begin{definition}\label{def:sp-pair}
We call $(\nu,\psi)$ of \eqref{eq:a1}  the \emph{spectral pair} of the operator $H$.
\end{definition}

\subsection{Wronskian notation}
We will use a Wronskian definition adapted to the differential equation \eqref{eq:init_diff_eq}. For $\bbC^2$-valued functions $F$ and $G$ on $\bbR_+$ we write
\begin{equation}
[F,G]=\jap{\eps F',G}_{\bbC^2}-\jap{\eps F,G'}_{\bbC^2}.
\label{eq:wr-def}
\end{equation}
Here the Euclidean inner product $\jap{a,b}_{\bbC^2}$ is linear in $a$ and anti-linear in $b$. 
If both $F$ and $G$ are solutions of \eqref{eq:init_diff_eq} \emph{with real $\lambda$}, then $[F,G]$ is independent of $x$. A~calculation shows that the Wronskian of any functions $F$ and $G$ at zero can be computed in terms of our boundary operators $L_\alpha$ and $L_\alpha^\perp$ (see \eqref{eq:e30c}) as follows:
\begin{equation}
[F,G](0)
=\jap{\eps L_\alpha(F),L_\alpha^\perp(G)}
-\jap{\eps L_\alpha^\perp(F),L_\alpha(G)}.
\label{eq:wr4}
\end{equation}

\section{Main results}\label{sec:mr-b}

\subsection{The Jost solutions}
For $k\in\bbC\setminus\{0\}$, we consider $\bbC^2$-valued solutions $F$ to the eigenvalue equation
\begin{equation}
-\epsilon F''+QF=k^2F
\label{e30}
\end{equation}
for the operator $\bH$ on the positive half-line. Of course, this is just \eqref{eq:ev-eq2} written in vector notation. Observe that if $Q=0$, we have four linearly independent solutions 
\[
\ee^{\pm kx}e_-, \quad 
\ee^{\pm \ii kx}e_+
\]
(see \eqref{eq:a4aa} for the definition of $e_{\pm}$); 
it is important to note that $\epsilon e_{\pm}=\pm e_{\pm}$. The following theorem is proved in Section~\ref{sec.e}.

\begin{theorem}\label{thm.e1}
Let $q\in L^1(\bbR_+)$ and $k\in\bbC\setminus\{0\}$. There exist \emph{Jost solutions} $\oldF_{-1}$, $\oldF_{+1}$, $\oldF_{+\ii}$ and $\oldF_{-\ii}$ to the differential equation \eqref{e30} 
with the asymptotics
\begin{align}
\oldF_{\pm1}(x,k)&=\ee^{\pm kx}e_-+\smallO(\ee^{\pm kx}), \label{eq:F-1}
\\
\oldF_{\pm \ii}(x,k)&=\ee^{\pm \ii kx}e_++\smallO(\ee^{\pm \ii kx}),\label{eq:F+i}
\end{align}
as $x\to\infty$. The derivatives with respect to $x$ satisfy the asymptotics
\begin{align}
\oldF_{\pm 1}'(x,k)&=\pm k\ee^{\pm kx}e_-+\smallO(\ee^{\pm kx}), \label{eq:F-1-prime}
\\
\oldF_{\pm \ii}'(x,k)&=\pm \ii k\ee^{\pm \ii kx}e_++\smallO(\ee^{\pm \ii kx}), \label{eq:F+i-prime}
\end{align}
as $x\to\infty$. 
\end{theorem}
We will call $\oldF_{\pm1}$ \emph{exponential Jost solutions} and $\oldF_{\pm\ii}$ \emph{oscillatory Jost solutions}.

\begin{remark} $\,$ \nopagebreak
\begin{enumerate}[1.]
\item
The Jost solutions $\oldF_{-1}$, $\oldF_{+1}$, $\oldF_{-\ii}$, $\oldF_{+\ii}$ are linearly independent and so any vector solution of \eqref{e30} is a linear combination of them. 
\item
Let $k>0$; observe that while $\oldF_{-1}$ is uniquely defined by the asymptotics \eqref{eq:F-1}, the other solutions are not. For example, for any $c\in\bbC$ the solution $\oldF_{+\ii}+c\oldF_{-1}$ satisfies the same asymptotics as $\oldF_{+\ii}$, because $c\oldF_{-1}$ can be absorbed into the error term of $\oldF_{+\ii}$. This non-uniqueness is a new feature of our set-up, compared with the classical self-adjoint theory. Because of this, our Jost solutions $\oldF_{j}$, $j\in\{\pm1,\pm\ii\}$ should be more accurately described as \emph{classes} of solutions satisfying the asymptotics stated in the theorem. We will occasionally use notation of the type $\myF\in \oldF_{+\ii}$ to mean that the function $F$ satisfies the differential equation \eqref{e30} and the asymptotics \eqref{eq:F+i}, \eqref{eq:F+i-prime}.
\item
We note that our proof of Theorem~\ref{thm.e1} in Section~\ref{sec.e} does not depend on the special structure \eqref{eq:a2} of the potential $Q$ or even on $Q$ being Hermitian. The only thing that matters is the integrability of $Q$. 
\item
The system \eqref{e30} bears a structural resemblance to a one-dimensional Dirac system since the Pauli matrix $\epsilon$ is indefinite. However, unlike first-order Dirac systems, our operator is governed by the second-order term $-\epsilon \frac{\dd^2}{\dd x^2}$. Since the matrix $\epsilon$ has eigenvalues $\pm 1$, the equation for $Q=0$ simultaneously admits both oscillatory ($\ee^{\pm \ii kx}$) and exponential ($\ee^{\pm kx}$) solutions for the same spectral parameter $k$. This coexistence of ``hyperbolic'' and ``elliptic'' modes distinguishes our case from standard Dirac theory and necessitates specific analysis developed in Section~\ref{sec.e}.
\end{enumerate}
\end{remark}

Returning to $Q$ of the special structure \eqref{eq:a2}, we will see that, in an important particular case, one can reduce considerations to a scalar equation. 

\begin{proposition}\label{cr.e3}
Let $q\in L^1(\bbR_+)$, let $Q$ be as in \eqref{eq:a2} and $k>0$. Then there exists a unique solution $\newe=\newe(x,k)$ to the anti-linear eigenvalue equation 
\begin{equation}
-\newe''+q\newe=-k^2\overline{\newe}
\label{e48}
\end{equation}
(note the complex conjugation in the right-hand side!) 
satisfying the asymptotics
\begin{equation}
\newe(x,k)=\ee^{-kx}+\smallO(\ee^{-kx}) \quad\text{as}\quad x\to\infty.
\label{e49}
\end{equation}
The Jost solutions $\oldF_{-1}(x,k)$ and $\oldF_{+\ii}(x,\ii k)$ of Theorem~\ref{thm.e1} can be expressed as
\begin{equation}
\oldF_{-1}(x,k)=\begin{pmatrix}\overline{\newe(x,k)}\\ -\newe(x,k)\end{pmatrix}, 
\qquad 
\oldF_{+\ii}(x,\ii k)=\begin{pmatrix}\overline{\newe(x,k)}\\ \newe(x,k)\end{pmatrix}. 
\label{e50}
\end{equation}
The point $k^2>0$ is an eigenvalue of $\bH$ if and only if $\newe$ satisfies the boundary condition
\begin{equation}
\ell_\alpha^\perp(\newe)=0.
\label{e6a}
\end{equation}
\end{proposition}
Proposition~\ref{cr.e3} is proved at the end of Section~\ref{sec.e}. 
We will call $\newe$ the \emph{scalar Jost solution}.

\subsection{The spectrum of $\bH$} 
First we describe the spectrum of $\bH$. 
\begin{theorem}\label{thm.a1}
Let $q\in L^1(\bbR_+)$ and let $Q$ be as in \eqref{eq:a2}.
\begin{enumerate}[\rm (i)]
\item
The non-zero eigenvalues of $\bH$ are simple. If infinite in number, these eigenvalues form a sequence converging to zero. If $\lambda=0$ is an eigenvalue of $\bH$, it has multiplicity two. 
\item
The absolutely continuous spectrum of $\bH$ coincides with $\bbR$ and has multiplicity one. For 
$\lambda\in\bbR\setminus(\sigma_{\mathrm p}(\bH)\cup\{0\})$, the Radon--Nikodym derivative $\dd\Sigma(\lambda)/\dd\lambda$ is a rank one matrix which is continuous  in $\lambda\in\bbR\setminus(\sigma_{\mathrm p}(\bH)\cup\{0\})$.
\item
The singular continuous spectrum of $\bH$ is absent. 
\end{enumerate}
\end{theorem}

Our aim is to describe
\begin{itemize}
\item
the density $\dd\nu(\lambda)/\dd\lambda$ and $\psi(\lambda)$, if $\lambda\not=0$ is not a point mass of $\nu$,
\item
$\nu(\{\lambda\})$ and $\psi(\lambda)$, if $\lambda\not=0$ is a point mass of $\nu$,
\end{itemize}
in terms of the Jost solutions $\oldF_{-1}$, $\oldF_{+\ii}$, with $k=\sqrt{\abs{\lambda}}>0$.

\subsection{Formulas for the spectral pair $(\nu,\psi)$}
Now we are ready to state our explicit formulas for $(\nu,\psi)$.  
Since $\nu$ is even and $\psi$ is odd, it suffices to characterise $(\nu,\psi)$ on the positive semi-axis. We recall the last claim of Proposition~\ref{cr.e3}: $k^2>0$ is an eigenvalue of $\bH$ if and only if $\ell_\alpha^\perp(\newe)=0$.

\begin{theorem}\label{thm.5.3}
Let $q\in L^1(\bbR_+)$ and let the Jost solutions $\oldF_{+\ii}(x,k)$, $\oldF_{-1}(x,k)$, $\newe(x,k)$ for $k>0$ be as described in Theorem~\ref{thm.e1} and Proposition~\ref{cr.e3}.
\begin{enumerate}[\rm(i)]
\item
Suppose $\lambda=k^2$ is not an eigenvalue of $\bH$. 
Then
\begin{equation}
\frac{\dd\nu}{\dd\lambda}(\lambda)=\frac{2k}{\pi}
\frac{\abs{\ell_\alpha^\perp(\newe)}^2}{\abs{\det\{L_\alpha^\perp(\oldF_{+\ii}),L_\alpha^\perp(\oldF_{-1})\}}^2},
\quad
\psi(\lambda)=
\frac{\ell_\alpha^\perp(\newe)}{\overline{\ell_\alpha^\perp(\newe)}}.
\label{ee5}
\end{equation}
The determinant in the denominator is unambiguously defined (i.e. it does not depend on the choice of the Jost solution $\oldF_{+\ii}$) and does not vanish. 
\item 
Suppose $\lambda=k^2$ is an eigenvalue of $\bH$. Then $\ell_{\alpha}(\newe)\neq0$ and we have
\begin{equation}
\nu(\{\lambda\})=
\frac{\abs{\ell_\alpha(\newe)}^2}{2\norm{\newe}^2},
\quad
\psi(\lambda)=
-\frac{\overline{\ell_\alpha(\newe)}}{\ell_\alpha(\newe)}.
\label{eq:nsa8}
\end{equation}
\end{enumerate}
\end{theorem}

In particular, we see from \eqref{ee5} and \eqref{eq:nsa8} that $\abs{\psi(\lambda)}=1$ for $\nu$-a.e. $\lambda>0$, which implies the simplicity of spectrum of $\bH$ on $\bbR_{+}$ by \cite[Theorem~1.5]{PS3}.

\begin{remark}
Similarly to  \eqref{eq:sa1}, formula \eqref{eq:nsa8} does not really require $q$ to be integrable. It holds for any eigenvalue of $\bH$ with $\newe$ being the unique (up to normalisation) $L^2$-solution of the anti-linear eigenvalue equation \eqref{e48}, see \cite[Theorem~1.13]{PS3}.
\end{remark}

\subsection{Scattering theory perspective}
Here we give a version of Theorem~\ref{thm.5.3} in the spirit of \eqref{X9}. 
For $\lambda>0$, let $\Phi$ be the $2\times2$ matrix-valued solution of the eigenvalue equation \eqref{eq:init_diff_eq}
defined by the initial conditions \eqref{eq:phi,theta_bc_cond}. Let us denote the columns of $\Phi$ by $\Phi_1$ and $\Phi_2$. In other words, $\Phi_1$ and $\Phi_2$ are $\bbC^2$-valued solutions of the eigenvalue equation \eqref{eq:init_diff_eq} with the Cauchy data
\begin{equation}
\left\{
\begin{aligned}
L_\alpha(\Phi_1)&=-e_1,
\\
L_\alpha^\perp(\Phi_1)&=0,
\end{aligned}
\right.
\quad
\left\{
\begin{aligned}
L_\alpha(\Phi_2)&=-e_2,
\\
L_\alpha^\perp(\Phi_2)&=0.
\end{aligned}
\right.
\label{eq:wr5}
\end{equation}
Since the Jost solutions $\oldF_{\pm1}$, $\oldF_{\pm\ii}$ (with $k=\sqrt{\lambda}>0$) form a basis, we can expand $\Phi_1$, $\Phi_2$ in this basis:
\begin{align}
\Phi_1&=\gamma^{1}_{-1}\oldF_{-1}+\gamma^{1}_{+\ii}\oldF_{+\ii}+\gamma^{1}_{-\ii}\oldF_{-\ii}+\gamma^{1}_{+1}\oldF_{+1},
\label{eq:ge1}
\\
\Phi_2&=\gamma^{2}_{-1}\oldF_{-1}+\gamma^{2}_{+\ii}\oldF_{+\ii}+\gamma^{2}_{-\ii}\oldF_{-\ii}+\gamma^{2}_{+1}\oldF_{+1}.
\label{eq:ge2}
\end{align}
Observe that \eqref{eq:ge1}, \eqref{eq:ge2} assume that we have made a \emph{choice} of representatives of the solution classes $\oldF_{+\ii}$, $\oldF_{-\ii}$, $\oldF_{+1}$. The coefficients $\gamma_{+1}^{1}$ and $\gamma_{+1}^{2}$ of the growing solution $\oldF_{+1}$ are unambiguously defined, while the rest of the coefficients depend on the choice of the three solutions $\oldF_{+\ii}$, $\oldF_{-\ii}$, $\oldF_{+1}$. However, it turns out that some \emph{combinations} of the coefficients $\gamma$ are unambiguously defined and it is possible to express the spectral pair $(\nu,\psi)$ in terms of these combinations. 

\begin{theorem}\label{thm:a4}
Let $q\in L^1(\bbR_+)$, $k>0$ and $\lambda=k^2$. 
\begin{enumerate}[\rm (i)]
\item
Assume $\lambda$ is not an eigenvalue of $\bH$. Then 
\begin{equation}
\gamma_{+1}^2=-\overline{\gamma_{+1}^1}\not=0
\label{eq:ge3}
\end{equation}
and the expression $\gamma_{+1}^2\gamma_{+\ii}^1-\gamma_{+1}^1\gamma_{+\ii}^2$ is unambiguously defined and non-zero. 
The spectral pair $(\nu,\psi)$ can be expressed as
\[
\frac{\dd\nu}{\dd\lambda}(\lambda)=\frac1{8\pi k}\frac{\abs{\gamma_{+1}^1}^2}{\abs{\gamma_{+1}^2\gamma_{+\ii}^1-\gamma_{+1}^1\gamma_{+\ii}^2}^2},
\quad
\psi(\lambda)=\frac{\overline{\gamma_{+1}^1}}{\gamma_{+1}^1}.
\]
\item
Assume $\lambda$ is an eigenvalue of $\bH$. Then $\gamma_{+1}^1=\gamma_{+1}^2=0$, 
the coefficients $\gamma_{\pm\ii}^1$, $\gamma_{\pm\ii}^2$ are unambiguously defined and 
\[
\abs{\gamma_{+\ii}^1}=\abs{\gamma_{-\ii}^1}=
\abs{\gamma_{+\ii}^2}=\abs{\gamma_{-\ii}^2}\not=0
\]
and
\[
\gamma_{+\ii}^2=\overline{\gamma_{-\ii}^1},
\quad
\gamma_{-\ii}^2=\overline{\gamma_{+\ii}^1}.
\]
The expression $\gamma_{-1}^1{\gamma_{-\ii}^2}-\gamma_{-\ii}^1\gamma_{-1}^2$ is unambiguously defined and non-zero.
The spectral pair $(\nu,\psi)$ can be expressed as
\[
\nu(\{\lambda\})=
\frac{1}{2\norm{\newe}^2}
\frac{\abs{\gamma_{+\ii}^1}^2}{\abs{\gamma_{-1}^1{\gamma_{-\ii}^2}-\gamma_{-\ii}^1\gamma_{-1}^2}^2}, 
\quad
\psi(\lambda)
=-\frac{{\gamma_{-\ii}^2}}{\gamma_{-\ii}^1}.
\]
\end{enumerate}
\end{theorem}

\subsection{The Born approximation for the spectral pair}
In order to provide some intuition into the spectral pair, in Section~\ref{sec:Born} we compute the linearised forms of $\nu$ and $\psi$ for small $q$;  this is known as Born approximation in scattering theory. Here for simplicity we only display these formulas for the Dirichlet case $\alpha=\infty$. In the classical self-adjoint case with real-valued $q\in L^{1}(\bR_{+})$, the density \eqref{eq:sa2} satisfies
\begin{equation}
\frac{\dd\sigma}{\dd\lambda}(\lambda)
=\frac{1}{\pi}\left(k-\int_0^\infty \sin(2ky)q(y)\dd y\right)+\calO(\norm{q}_{L^1}^2) 
\label{eq:ba1}
\end{equation}
for any $\lambda=k^{2}>0$. 

\begin{proposition}\label{prop:Born}
Let $\alpha=\infty$. As $\norm{q}_{L^1}\to0$, for any $\lambda=k^{2}>0$ we have 
\begin{equation}
\frac{\dd\nu}{\dd\lambda}(\lambda)=
\frac{1}{2\pi}\left(k-\int_0^\infty \sin(2ky)\Re q(y)\dd y\right)+\calO(\norm{q}_{L^1}^2)
\label{eq:ba2}
\end{equation}
and
\begin{equation}
\psi(\lambda)=1+\frac{2\ii}k\int_0^\infty  \ee^{-ky} \sin(ky)\Im q(y)\dd y+\calO(\norm{q}_{L^1}^2).
\label{eq:ba3}
\end{equation}
\end{proposition}
For real $q$, formula \eqref{eq:ba2} reduces to \eqref{eq:ba1}, up to a normalisation factor of $1/2$. This is in agreement with \cite[Theorem~1.7]{PS3}, which states that $\dd\sigma(\lambda)=2\dd\nu(\lambda)$ for real $q$. We observe that \eqref{eq:ba3} depends only on the imaginary part of $q$ and contains the exponential factor $\ee^{-ky}$.

\section{Construction of the Jost solutions}\label{sec.e}

\subsection{Preliminaries. Sectors in the $k$-plane and Stokes lines}
The aim of this section is to prove Theorem~\ref{thm.e1}. Our proof does not rely on the special form \eqref{eq:a2} of the $2\times2$ matrix-valued potential $Q$ or even on $Q$ being Hermitian. We only need the integrability of $Q$, i.e.
\[
\int_0^\infty\norm{Q(x)}\dd x<\infty.
\]
Here and below, $\norm{\cdot}$ means the operator norm on the set of $2\times2$ matrices, i.e. these matrices are considered as linear operators on $\bbC^2$ with the Euclidean norm. 

Broadly speaking, we follow the standard strategy of converting the differential equation \eqref{e30} into an integral equation. However, we face the following challenge. Consider the solutions $\ee^{\pm kx}e_-$,  $\ee^{\pm \ii kx}e_+$ of the free equation ($Q=0$). The hierarchy of these solutions in terms of the rate of growth/decay is governed by the real parts of the four exponents $\pm k$, $\pm\ii k$. This hierarchy changes as $k$ crosses any of the four \emph{Stokes lines} 
\[
\Re k=0, \quad \Im k=0, \quad \Re k=\Im k, \quad \Re k=-\Im k
\]
in the complex plane. These Stokes lines split the $k$-plane into eight sectors, with a different hierarchy of Jost solutions in each of these sectors. As a~consequence, we have to set up our integral equations in a different way in each sector. In other words, the formula for the Green function for each of the four Jost solutions has to be adapted to the sector.

Thus, we have four Jost solutions and eight sectors in the complex plane; it is a challenge to describe all corresponding integral equations systematically. We denote
\[
\begin{aligned}
S_0&=\{k\in\bbC\setminus\{0\}: 0<\arg k<\pi/4\},
\\
\overline{S_0}&=\{k\in\bbC\setminus\{0\}: 0\leq\arg k\leq \pi/4\}.
\end{aligned}
\]
Our strategy consists of the following steps. 
\begin{itemize}
\item
Construct all four Jost solutions for $k$ in the open sector $S_0$. This has to be done differently for each of the four Jost solutions. 
\item
Consider the boundary Stokes lines of $S_0$ and establish the existence of Jost solutions in $\overline{S_0}$. 
\item
By using the transformation $k\mapsto\overline{k}$, extend this construction to the sector $-\pi/4\leq\arg k\leq \pi/4$. 
\item
By using the transformation $k\mapsto \ii k$, extend this construction to the sector $-\pi/4\leq\arg k\leq 3\pi/4$. 
\item
By using the transformation $k\mapsto -k$, extend this construction to all $k\not=0$. 
\end{itemize}

The key difficulty is the first step of the proof. The subtle aspect of it (compared to the standard case $\epsilon=1$) is that we are forced to work with the integral equations that are in general of Fredholm type but not of Volterra type (see \eqref{eq:e65}). This aspect makes our analysis structurally closer to scattering theory in dimension $d\geq2$ than to the standard one-dimensional theory. We will follow the strategy of \cite{YaO}, which is based on an idea from the book \cite{CodLev}, see Problem~29 of Chapter~3 therein.

\subsection{Setting up integral equations in $\overline{S_0}$}
For $k\in\overline{S_0}$ we define the components of Green's functions as follows:
\begin{align*}
A(x)&=\frac{\ee^{-k\abs{x}}}{2k}P_-,
&&A_{\pm}(x)=\mp\frac{\sinh(kx)}{k}\chi_{\pm}(x)P_-,
\\
B(x)&=\frac{\ee^{\ii k\abs{x}}}{2\ii k}P_+,
&&B_{\pm}(x)=\pm\frac{\sin(kx)}k\chi_{\pm}(x)P_+,
\end{align*}
where $\chi_+$ (resp. $\chi_-$) is the characteristic function of the positive (resp. negative) semi-axis and $P_\pm$ are the projections defined in \eqref{eq:proj}. Next, we define 
\begin{equation}
\begin{aligned}
L_{-1}(x)&=A_-(x)+B_-(x), 
\\
L_{+1}(x)&=A(x)+B_+(x),
\\
L_{+\ii}(x)&=A(x)+B_-(x), 
\\
L_{-\ii}(x)&=A(x)+B(x).
\end{aligned}
\label{eq:e58}
\end{equation}
We fix an additional parameter $r\geq0$, whose purpose will be explained later. Let us set up the integral equations for the Jost solutions. These are integral equations on the interval $[r,\infty)$; since we are mostly interested in the asymptotic behaviour of the Jost solutions at infinity, restricting $x$ to $[r,\infty)$ is not a serious constraint. 
We will be looking for solutions ${\myF}_{j}$, $j\in\{\pm1,\pm\ii\}$,  to the integral equations
\begin{align}
{\myF}_{\pm1}(x)&= \ee^{\pm kx} e_- + \int_{r}^{\infty} L_{\pm1}(x-y) Q(y) {\myF}_{\pm1}(y) \dd y,\quad x\geq r, 
\label{eq:t1}
\\
{\myF}_{\pm\ii}(x)&= \ee^{\pm\ii kx} e_+ + \int_{r}^{\infty} L_{\pm\ii}(x-y) Q(y) {\myF}_{\pm\ii}(y) \dd y, 
\quad x\geq r,
\label{eq:t2}
\end{align}
where the vectors $e_{\pm}$ are defined in~\eqref{eq:a4aa}. If necessary, we will indicate the dependence on $k$ and $r$ by writing ${\myF}_{j}(x,k)$ or ${\myF}_{j}(x,k;r)$.

\begin{remark*}$\,$ \nopagebreak
\begin{enumerate}[1.]
\item
In the classical scalar self-adjoint case (i.e. for the Jost solution \eqref{eq:sa4}) the integration in the corresponding integral equation is from $x$ to infinity. 
\item
We will prove the existence of solutions ${\myF}_{j}$ to the integral equations \eqref{eq:t1}, \eqref{eq:t2} for $k\in\overline{S_0}$ and $r$ sufficiently large. The purpose of the parameter $r$ is to make the norm of the integral operator in the right-hand side smaller than one to ensure the convergence of the Neumann series. This will be explained precisely below. 
\item
Even though \eqref{eq:t1}, \eqref{eq:t2} are integral equations on $[r,\infty)$, we will soon see that all $\myF_{j}$ satisfy the differential equation \eqref{e30}. Thus, $\myF_{j}$ can be uniquely extended to $[0,\infty)$. 
\item
We will show that ${\myF}_{j}\in\oldF_{j}$ (i.e. ${\myF}_{j}$ satisfy the asymptotics \eqref{eq:F-1}--\eqref{eq:F+i-prime}) in the open sector $S_0$. On the boundary Stokes rays $k>0$ and $\arg k=\pi/4$ the situation is more complicated. In general, ${\myF}_{j}$ is a linear combination of a solution from the class $\oldF_{j}$ and some other solution.  However, we mostly care about the solutions with $j=-1$ and $j=+\ii$ (these form the basis of the space of all exponentially decaying solutions for $k\in S_0$, and these enter our main Theorem~\ref{thm.5.3}) when $k$ approaches the real line from above, and the good news is that for $k>0$, we will prove that $\myF_{-1}\in\oldF_{-1}$ and $\myF_{+\ii}\in\oldF_{+\ii}$. 
\end{enumerate}
\end{remark*}

\subsection{The choice of parameters $r\geq0$ and $\delta>0$ for the Neumann series}
In what follows, we fix $\delta>0$ and choose $r\geq0$ sufficiently large so that
\begin{equation}
\frac{2}{\delta} \int_{r}^{\infty} \norm{Q(y)} \dd y < \frac{1}{2}.
\label{e38}
\end{equation}
The constant $1/2$ on the right-hand side can be replaced by any positive number $<1$.
The expression on the left-hand side of \eqref{e38} is the upper bound on the norm of the integral operators in \eqref{eq:t1}, \eqref{eq:t2} for $\abs{k}>\delta$, after a suitable weighting (see \eqref{eq:t4}, \eqref{eq:t5}).
We will be able to establish the existence and uniqueness of solutions to \eqref{eq:t1} and \eqref{eq:t2} for $k\in\overline{S_0}$ and $\abs{k}>\delta$, as convergent Neumann series. Let us denote
\begin{align*}
S_\delta&=\{k\in\bbC: 0<\arg k<\pi/4, \quad \abs{k}>\delta\},
\\
\overline{S_\delta}&=\{k\in\bbC: 0\leq\arg k\leq\pi/4, \quad \abs{k}\geq\delta\}.
\end{align*}

\subsection{$\oldF_{j}$ and $\myF_{j}$: classes and representatives for Jost solutions}

We summarise our notational conventions. As already discussed, $\oldF_{j}(x,k)$ denote \emph{classes} of solutions satisfying the asymptotics stated in Theorem~\ref{thm.e1}. 

We denote by $\myF_{j}(x,k;r)$ the \emph{uniquely defined} solution of the corresponding integral equation \eqref{eq:t1}--\eqref{eq:t2}. The solution $\myF_{j}(x,k;r)$ in general depends on $r$. 

\begin{remark*}
For $k\in\overline{S_0}$, the integral equation \eqref{eq:t1} for $\myF_{-1}(x,k;r)$ is Volterra. It follows that for this range of $k$, the solution $\myF_{-1}(x,k;r)$ is actually independent of the choice of $r$. This aligns with the fact that for this range of $k$, the class of solutions $\oldF_{-1}$ is a one-point set. 
\end{remark*}

\subsection{Existence of Jost solutions in $S_0$}
It will be convenient to define the auxiliary functions $w_{\pm1}$ and $w_{\pm\ii}$ by 
\begin{equation}
{\myF}_{\pm1}(x)=\ee^{\pm kx}w_{\pm1}(x), 
\quad
{\myF}_{\pm\ii}(x)=\ee^{\pm\ii kx}w_{\pm\ii}(x), 
\label{eq:t3}
\end{equation}
and the corresponding kernels $K_{\pm1}$ and $K_{\pm\ii}$ by 
\[
\begin{aligned}
K_{-1}(x)&=\ee^{kx}L_{-1}(x)=\ee^{kx}(A_-(x)+B_-(x)),
\\
K_{+1}(x)&=\ee^{-kx}L_{+1}(x)=\ee^{-kx}(A(x)+B_+(x)),
\\
K_{+\ii}(x)&=\ee^{-\ii kx}L_{+\ii}(x)=\ee^{-\ii kx}(A(x)+B_-(x)),
\\
K_{-\ii}(x)&=\ee^{\ii kx}L_{-\ii}(x)=\ee^{\ii kx}(A(x)+B(x)).
\end{aligned}
\]
For brevity, we suppress the dependence of $w_{j}$ and $K_{j}$ on $k$ and $r$.

The integral equations \eqref{eq:t1}--\eqref{eq:t2} for $\myF_{j}$ are equivalent to the integral equations 
\begin{align}
w_{\pm1}(x)&= e_- + \int_{r}^{\infty} K_{\pm1}(x-y) Q(y) w_{\pm1}(y) \dd y, 
\label{eq:t4}
\\
w_{\pm\ii}(x)&= e_+ + \int_{r}^{\infty} K_{\pm\ii}(x-y) Q(y) w_{\pm\ii}(y) \dd y.
\label{eq:t5}
\end{align}
In the rest of this section, we prove the existence of solutions $w_{j}$ to these integral equations.
We denote $x_+=\max\{x,0\}$ and $K_{j}'=\frac{\partial}{\partial x}K_{j}$. 
\begin{lemma}\label{lma:t1}
For all indices $j\in\{\pm1,\pm\ii\}$, all $x\in\bbR$ and all $k\in\overline{S_0}$, we have the estimates
\begin{equation}
\norm{K_{j}(x)}\leq \frac2{\abs{k}}
\quad\mbox{ and }\quad
\norm{K_{j}'(x)}\leq 4.
\label{e45}
\end{equation}
If $k\in S_0$, we have the estimate 
\begin{equation}
\norm{K_{j}(x)}+\norm{K_{j}'(x)}\leq C\ee^{-d x_+},
\label{eq:t8}
\end{equation}
where $C=C(k)$ and $d=d(k)$ are positive constants independent of $j$ and $x$.
\end{lemma}

\begin{proof}
We denote 
\[
k=a+\ii b,
\]
where $0\leq b\leq a$ in our sector $\overline{S_0}$. Each of the kernels $K_{j}$ is a sum of several exponential terms. The proof will follow if we check that the real parts of all exponents are (i) non-positive for $k\in\overline{S_0}$, $x\in\bbR$ and (ii) strictly negative for $k\in S_0$, $x>0$. We display the results in the form of four tables (one of each index $j$) for the ease of checking.  

\textbf{Kernel $K_{-1}$:}
\[
K_{-1}(x) =\frac1k \ee^{kx} \left( \sinh(kx) P_- - \sin(kx) P_+ \right) \chi_-(x).
\]

\begin{center}
    \renewcommand{\arraystretch}{1.2} 
    \begin{tabular}{|c|c|c|c|}
    \hline
    \textbf{Component} & \textbf{Domain of $x$} & \textbf{Total Exponent} & \textbf{Real Part ($\le 0$)} \\ \hline
    $P_-$-term ($A_-$) & $x < 0$ & \alignx{(k \pm k)} & \alignx{(a \pm a)} \\ \hline
    $P_+$-term ($B_-$) & $x < 0$ & \alignx{(k \pm \ii k)} & \alignx{(a \mp b)} \\ \hline
    \end{tabular}
\end{center}
\smallskip

\textbf{Kernel $K_{+\ii}$:}
\begin{align*}
K_{+\ii}(x)&= 
\frac1{2k}\ee^{-\ii kx}\left(\ee^{-k\abs{x}}P_--2\sin(kx)P_+\chi_-(x)\right)
\\
&=
\begin{cases} 
\frac1{2k}{\ee^{-(k+\ii k)x}} P_- & \text{for } x > 0,
\\[4pt]
\frac{1}{2k}{\ee^{(k-\ii k)x}} P_- - \frac1{k}{\ee^{-\ii kx}\sin(kx)} P_+ & \text{for } x < 0.
\end{cases}
\end{align*}

\begin{center}
    \renewcommand{\arraystretch}{1.2}
    \begin{tabular}{|c|c|c|c|}
    \hline
    \textbf{Component} & \textbf{Domain of $x$} & \textbf{Total Exponent} & \textbf{Real Part ($\le 0$)} \\ \hline
    $P_-$-term ($A$) & $x > 0$ & \alignx{-(k+\ii k)} & \alignx{-(a-b)} \\ \cline{2-4}
                     & $x < 0$ & \alignx{(k-\ii k)} & \alignx{(a+b)} \\ \hline
    $P_+$-term ($B_-$) & $x < 0$ & \alignx{(-\ii k\pm\ii k)} & \alignx{(b\mp b)} \\ \hline
    \end{tabular}
\end{center}
\smallskip
    
\textbf{Kernel $K_{-\ii}$:}
\begin{align*}
K_{-\ii}(x)&=\frac1{2k}\ee^{\ii kx}\left(\ee^{-k\abs{x}}P_--\ii e^{\ii k\abs{x}}P_+\right)
\\
&= 
\begin{cases} 
\frac{1}{2k}\left(\ee^{-(k-\ii k)x} P_- -\ii \ee^{2\ii kx} P_+\right) & \text{for } x > 0,
\\[4pt]
\frac{1}{2k}\left(\ee^{(k+\ii k)x} P_- -\ii P_+\right) & \text{for } x < 0.
\end{cases}
\end{align*}

\begin{center}
    \renewcommand{\arraystretch}{1.2}
    \begin{tabular}{|c|c|c|c|}
    \hline
    \textbf{Component} & \textbf{Domain of $x$} & \textbf{Total Exponent} & \textbf{Real Part ($\le 0$)} \\ \hline
    $P_-$-term ($A$) & $x > 0$ & \alignx{-(k-\ii k)} & \alignx{-(a+b)} \\ \cline{2-4}
                     & $x < 0$ & \alignx{(k+\ii k)} & \alignx{(a-b)} \\ \hline
    $P_+$-term ($B$) & $x > 0$ & \alignx{2\ii k} & \alignx{-2b} \\ \cline{2-4}
                     & $x < 0$ & \makebox[2.2cm][r]{}\phantom{:}{0} & \makebox[2.2cm][r]{}\phantom{:}{0} \\ \hline
    \end{tabular}
\end{center}
\smallskip    
    
\textbf{Kernel $K_{+1}$:}    
\begin{align*}
K_{+1}(x)&=\frac1{2k}\ee^{-kx}\left(\ee^{-k\abs{x}}P_-+2\sin(kx)P_+\chi_+(x)\right)
\\    
&= 
\begin{cases}\frac1{2k}\ee^{-2kx} P_- + \frac1k\ee^{-kx}\sin(kx) P_+ & \text{for } x > 0,
\\[4pt]
\frac{1}{2k} P_- & \text{for } x < 0.
\end{cases}
\end{align*}

\begin{center}
    \renewcommand{\arraystretch}{1.2}
    \begin{tabular}{|c|c|c|c|}
    \hline
    \textbf{Component} & \textbf{Domain of $x$} & \textbf{Total Exponent} & \textbf{Real Part ($\le 0$)} \\ \hline
    $P_-$-term ($A$) & $x > 0$ & \alignx{-2k} & \alignx{-2a} \\ \cline{2-4}
                     & $x < 0$ & \makebox[2.2cm][r]{}\phantom{:}{0} & \makebox[2.2cm][r]{}\phantom{:}{0} \\ \hline
    $P_+$-term ($B_+$) & $x > 0$ & \alignx{-(k \pm \ii k)} & \alignx{-(a \mp b)} \\ \hline
    \end{tabular}
\end{center}
\smallskip

The proof follows by inspection of all the four tables above.
\end{proof}

\begin{lemma}
Let $\delta>0$ and $r\geq0$ satisfy \eqref{e38}. Then for any $k\in\overline{S_\delta}$, 
the integral equations \eqref{eq:t4} (resp. \eqref{eq:t5}) have unique solutions $w_{\pm1}$ (resp. $w_{\pm\ii}$) in $L^\infty((r,\infty);\bbC^2)$.
\end{lemma}

\begin{proof}
For definiteness, consider \eqref{eq:t4}. Let $\calK_{\pm1}$ be the integral operator acting on $w_{\pm1}$ in the right-hand side of \eqref{eq:t4}; then this integral equation can be written as
\begin{equation}
w_{\pm1}=e_-+\calK_{\pm1}w_{\pm1}.
\label{eq:t6}
\end{equation}
Using \eqref{e45} and \eqref{e38}, we estimate the norm of this operator 
\[
\norm{\calK_{\pm1}}_{L^\infty\to L^\infty}
=
\sup_{x\geq r}\int_r^\infty \norm{K_{\pm1}(x-y) Q(y)}\dd y
\]
as
\begin{equation}
\norm{\calK_{\pm1}}_{L^\infty\to L^\infty}
\leq
\frac2{\abs{k}}\int_r^\infty\norm{Q(y)}\dd y
\leq 
\frac2{\delta}\int_r^\infty\norm{Q(y)}\dd y
<\frac12,
\label{e57}
\end{equation}
and so the integral equation \eqref{eq:t6} has a unique solution as the Neumann series
\begin{equation}
w_{\pm1}=\sum_{m=0}^\infty \calK_{\pm1}^m e_{-} , 
\label{eq:ns}
\end{equation}
convergent uniformly in $k\in\overline{S_\delta}$.
\end{proof}

Let us now define ${\myF}_{j}={\myF}_{j}(x,k;r)$ by \eqref{eq:t3}, where $k\in\overline{S_\delta}$ and the solutions $w_{j}$ are as in the previous lemma. By construction, ${\myF}_j$ satisfy the integral equations \eqref{eq:t1} and \eqref{eq:t2}. The next (standard) lemma shows that ${\myF}_{j}$ also satisfy the differential equation \eqref{e30}. 

\begin{lemma}
For all indices $j\in\{\pm1,\pm\ii\}$, the kernels $L_j$ defined by~\eqref{eq:e58} satisfy
\[
\left(-\epsilon \frac{\dd^2}{\dd x^2} - k^2\right) L_{j}(x) = -\delta(x),
\]
where $\delta$ is the delta-function. Thus, the solutions ${\myF}_{j}$ to the integral equations  \eqref{eq:t1} and \eqref{eq:t2} satisfy the differential equation \eqref{e30} on $(r,\infty)$. 
\end{lemma}
\begin{proof}
Follows directly by checking the relations
\begin{align*}
-\epsilon A''(x)-k^2 A(x) &= -\delta(x)P_-,
\quad
-\epsilon A_\pm''(x)-k^2 A_\pm(x) = -\delta(x)P_-,
\\
-\epsilon B''(x)-k^2 B(x) &= -\delta(x)P_+,
\quad
-\epsilon B_\pm''(x)-k^2 B_\pm(x) = -\delta(x)P_+.
\end{align*}
Here it is important that $\epsilon P_+=P_+$ and $\epsilon P_-=-P_-$. 
\end{proof}

Since the functions ${\myF}_{j}$ satisfy the differential equation \eqref{e30} on $(r,\infty)$, we will extend them to the whole of $[0,\infty)$ as solutions to the same differential equation. 

\subsection{The asymptotics of solutions $\myF_{j}$ at infinity}

First for the ease of reference we display an elementary lemma. 
\begin{lemma}\label{lma:t4}
Let $d>0$. Then 
\[
\lim_{x\to\infty}\int_0^\infty \ee^{-d(x-y)_+}\norm{Q(y)}\dd y=0.
\]
\end{lemma}
\begin{proof} 
Split the domain of integration into $y\leq x/2$ and $y>x/2$ and estimate
\begin{align*}
\int_0^{x/2}\ee^{-d(x-y)_+}\norm{Q(y)}\dd y&\leq \ee^{-d x/2}\int_{0}^\infty \norm{Q(y)} \dd y, 
\\
\int_{x/2}^\infty \ee^{-d (x-y)_+}\norm{Q(y)}\dd y&\leq \int_{x/2}^\infty\norm{Q(y)}\dd y.
\qedhere
\end{align*}
\end{proof}

\begin{lemma}\label{lma:t5}
Let $\delta>0$ and $r\geq0$ satisfy \eqref{e38}. Then for any $k\in S_\delta$ and all indices $j\in\{\pm1,\pm\ii\}$, the solutions $\myF_{j}$ satisfy the asymptotics \eqref{eq:F-1}--\eqref{eq:F+i-prime}. 
In the notation of Theorem~\ref{thm.e1}, we have $\myF_{j}(\cdot,k;r)\in\oldF_{j}(\cdot,k)$.
\end{lemma}
\begin{proof}
Let us first discuss the asymptotics \eqref{eq:F-1}--\eqref{eq:F+i} of the solutions $\myF_{j}(x,k)$; we will come to the derivatives \eqref{eq:F-1-prime}--\eqref{eq:F+i-prime} at the last step of the proof. The sought asymptotics is equivalent to 
\[
w_{\pm1}(x)\to e_-, \quad w_{\pm\ii}(x)\to e_+, \quad x\to\infty,
\]
see~\eqref{eq:t3}.
Let us use the integral equations \eqref{eq:t4}--\eqref{eq:t5}. 
Since we already know that $w_{j}\in L^\infty$, it suffices to check that 
\[
\lim_{x \to \infty} \int_{r}^{\infty} \norm{K_{j}(x-y)Q(y)} \dd y = 0. 
\]
This limiting relation follows from Lemma~\ref{lma:t4} and \eqref{eq:t8}. 

Next, let us establish the asymptotics \eqref{eq:F-1-prime}--\eqref{eq:F+i-prime} of the derivatives $\myF_{j}'(x,k)$. We need to prove that 
\begin{equation}
w_{j}'(x)\to 0, \quad x\to\infty.
\label{eq:t8a}
\end{equation}
By the integral equations \eqref{eq:t4} and \eqref{eq:t5}, we have
\[
w_{j}'(x) = \int_{r}^{\infty} K_{j}'(x-y) Q(y)w_{j}(y) \dd y.
\]
By Lemma~\ref{lma:t4} and \eqref{eq:t8}, we find 
\[
\lim_{x \to \infty} \int_{r}^{\infty} \Norm{K_{j}'(x-y)Q(y)} \dd y = 0.
\]
Since we already know that $w_{j}\in L^\infty$, this yields \eqref{eq:t8a}. 
\end{proof}

\begin{remark*}
It is instructive to consider the special case $j=-1$, which is particularly simple due to the Volterra structure of the integral kernel. In this special case Lemma~\ref{lma:t4} is not needed.
\end{remark*}

\noindent \textbf{Interim summary:}
At this stage of the proof, we have established Theorem~\ref{thm.e1} for each $k$ in the open sector $S_0$. Indeed, we have found a representative $\myF_{j}(\cdot,k;r)$ in each solution class $\oldF_{j}$.

Next we consider the boundary Stokes rays $\arg k=0$ (i.e. $k>0$) and $\arg k=\pi/4$.

\begin{lemma}\label{lma:t6} $\ $ \nopagebreak
\begin{enumerate}[\rm (i)]
\item
Let $k>0$. Then the solutions ${\myF}_{-1}$, ${\myF}_{+\ii}$ and ${\myF}_{+1}$ satisfy the asymptotics \eqref{eq:F-1}--\eqref{eq:F+i-prime}, while ${\myF}_{-\ii}$ satisfies
\begin{align}
{\myF}_{-\ii}(x,k)&=\ee^{-\ii kx}e_+ +\beta_1\ee^{\ii kx}e_++\smallO(1), 
\label{eq:t9}
\\
{\myF}_{-\ii}'(x,k)&=-\ii k\ee^{-\ii kx}e_++\ii \beta_1 k \ee^{\ii kx}e_++\smallO(1), 
\label{eq:t10}
\end{align}
as $x\to\infty$, with some $\beta_1\in\bbC$. 
\item
Let $\arg k=\pi/4$. Then the solutions ${\myF}_{-1}$ and ${\myF}_{-\ii}$ satisfy the asymptotics \eqref{eq:F-1}--\eqref{eq:F+i-prime}, while ${\myF}_{+\ii}$ and ${\myF}_{+1}$ satisfy
\begin{align}
{\myF}_{+\ii}(x,k)&=\ee^{\ii kx}e_++\beta_2\ee^{-kx}e_-+\smallO(\ee^{-kx}), 
\label{eq:t11}
\\
{\myF}_{+\ii}'(x,k)&=\ii k\ee^{\ii kx}e_+-\beta_2 k \ee^{-kx}e_-+\smallO(\ee^{-kx}), 
\label{eq:t12}
\\
{\myF}_{+1}(x,k)&=\ee^{kx}e_-+\beta_3\ee^{-\ii kx}e_++\smallO(\ee^{kx}),
\label{eq:t13}
\\
{\myF}_{+1}'(x,k)&=k\ee^{kx}e_--\ii \beta_3 k \ee^{-\ii kx}e_++\smallO(\ee^{kx}),
\label{eq:t14}
\end{align}
as $x\to\infty$, with some $\beta_2,\beta_3\in\bbC$.
\end{enumerate}
\end{lemma}
\begin{proof}

For the solution ${\myF}_{-1}$, the integral equation is of Volterra type, and so the same proof as in the previous lemma works. This proof only relied on the estimate \eqref{e45}, which holds true for $K_{-1}$ in the closed sector $\overline{S_0}$, including both boundary rays. In the rest of the proof we discuss $F_{\pm\ii}$ and $F_{+1}$.

\emph{Part (i):} Let $k>0$. 
Consider ${\myF}_{+\ii}$ and ${\myF}_{+1}$. Inspecting the tables in the proof of Lemma~\ref{lma:t1}, we find that 
\[
\norm{K_{+\ii}(x)}+\norm{K_{+\ii}'(x)}
+\norm{K_{+1}(x)}+\norm{K_{+1}'(x)}\leq C\ee^{-kx_+}, \quad 
x\in\bbR,
\]
with some $C=C(k)>0$, and therefore (cf. \eqref{eq:t8}) by the argument of the previous lemma we get the required asymptotics for ${\myF}_{+\ii}$ and ${\myF}_{+1}$. 

Let us consider ${\myF}_{-\ii}$. The integral equation for $w_{-\ii}$ is 
\begin{equation}
w_{-\ii}(x)=e_++\frac1{2k}\int_r^\infty \ee^{\ii k(x-y)}(\ee^{-k\abs{x-y}}P_--\ii\ee^{\ii k \abs{x-y}}P_+)Q(y)w_{-\ii}(y)\dd y.
\label{eq:t15}
\end{equation}
Let us split the integral on the right-hand side as $\int_r^x+\int_x^\infty$. The second integral contributes $\smallO(1)$ to the asymptotics. The first integral contains two terms: one with $P_-$ and one with $P_+$. The term with $P_-$ contributes $\smallO(1)$ to the asymptotics due to the exponential decay of the integral kernel (cf. \eqref{eq:t8}). Consider the remaining integral:
\[
\frac{-\ii}{2k}\int_r^x \ee^{2\ii k(x-y)}P_+Q(y)w_{-\ii}(y)\dd y
=
\frac{-\ii}{2k}\ee^{2\ii k x}\left(\int_r^\infty \ee^{-2\ii ky}P_+Q(y)w_{-\ii}(y)\dd y+\smallO(1)\right), 
\]
as $x\to\infty$. The integral in brackets converges absolutely. Due to the projection $P_+$ in the integrand, the integral is proportional to the vector $e_+$. 
Thus, we can write 
\[
\frac{-\ii}{2k}\int_r^\infty \ee^{-2\ii ky}P_+Q(y)w_{-\ii}(y)\dd y=\beta_1 e_+
\]
with some $\beta_1\in\bbC$. 
With this notation, we obtain 
\[
w_{-\ii}(x)=e_++\beta_1 \ee^{2\ii k x}e_++\smallO(1), \quad x\to\infty.
\]
Returning to the notation ${\myF}_{-\ii}$, we obtain the asymptotics \eqref{eq:t9}. 
Differentiating \eqref{eq:t15}, by the same methods we obtain the asymptotics \eqref{eq:t10} for the derivative. 

\emph{Part (ii):} Let $\arg k=\pi/4$; we set $k=a+\ii a$ with $a>0$.
In this part we will be brief as the arguments are analogous to those of part (i). 
For ${\myF}_{-\ii}$ we check 
\[
\norm{K_{-\ii}(x)}+\norm{K_{-\ii}'(x)}\leq C\ee^{-2ax_+}, \quad 
x\in\bbR,
\]
and so the respective asymptotics of \eqref{eq:F+i} and \eqref{eq:F+i-prime} follow. For ${\myF}_{+\ii}$, by the same pattern as in part (i), we find 
\begin{align*}
w_{+\ii}(x)&=e_+ +\int_r^x K_{+\ii}(x-y)Q(y)w_{+\ii}(y)\dd y+\smallO(1)
\\
&=e_+ +\frac1{2k}\int_r^x \ee^{-2\ii a(x-y)}P_-Q(y)w_{+\ii}(y)\dd y+\smallO(1)
\\
&=e_+ +\beta_2\ee^{-2\ii ax}e_-+\smallO(1), 
\end{align*}
as $x\to\infty$ with some $\beta_2\in\bbC$. Expressing this in terms of ${\myF}_{+\ii}$, we obtain \eqref{eq:t11}. 

Finally, consider ${\myF}_{+1}$. For $x>0$ we have 
\begin{align*}
K_{+1}(x)&=\frac1{2k}\ee^{-2kx}P_{-}+\frac1{2\ii k}\ee^{-(k-\ii k)x}P_{+}-\frac1{2\ii k}\ee^{-(k+\ii k)x}P_{+}
\\
&=-\frac1{2\ii k}\ee^{-2\ii ax}P_{+}+\calO(\ee^{-2ax}), 
\end{align*}
and therefore
\begin{align*}
w_{+1}(x)&=e_- +\int_r^x K_{+1}(x-y)Q(y)w_{+1}(y)\dd y+\smallO(1)
\\
&=e_- -\frac1{2\ii k}\int_r^x \ee^{-2\ii a(x-y)}P_+Q(y)w_{+1}(y)\dd y+\smallO(1)
\\
&=e_- +\beta_3\ee^{-2\ii ax}e_++\smallO(1), 
\end{align*}
as $x\to\infty$. Expressing this in terms of ${\myF}_{+1}$, we obtain \eqref{eq:t13}. Formulas \eqref{eq:t12} and \eqref{eq:t14} for the derivatives are deduced by the same approach.
\end{proof}

\subsection{Proof of Theorem~\ref{thm.e1} for $k$ in the closed sector $\overline{S_0}$.}
It suffices to prove the theorem for $k\in\overline{S_\delta}$ for any $\delta>0$. We fix $\delta>0$ and $r\geq0$ as in \eqref{e38}. We must exhibit at least one representative in each class $\oldF_{j}(\cdot,k)$ for our range of $k$.

By Lemma~\ref{lma:t5}, we have $\myF_{j}(\cdot,k;r)\in{\oldF}_{j}(\cdot,k)$ in $S_\delta$; it remains to consider the boundary Stokes rays $\arg k=0$ and $\arg k=\pi/4$. We use Lemma~\ref{lma:t6} and proceed step by step as follows (suppressing dependence on $k$ and $r$ for readability):
\begin{itemize}
\item
$\myF_{-1}\in{\oldF}_{-1}$ both for $\arg k=0$ and for $\arg k=\pi/4$;
\item
$\myF_{+\ii}\in{\oldF}_{+\ii}$ for $\arg k=0$ and ${\myF}_{+\ii}-\beta_2 {\myF}_{-1}\in \oldF_{+\ii}$ for $\arg k=\pi/4$;
\item
${\myF}_{-\ii}-\beta_1{\myF}_{+\ii}\in \oldF_{-\ii}$ for $\arg k=0$ and ${\myF}_{-\ii}\in\oldF_{-\ii}$ for $\arg k=\pi/4$;
\item
${\myF}_{+1}\in \oldF_{+1}$ for $\arg k=0$
and 
${\myF}_{+1}-\beta_3{\myF}_{-\ii}\in \oldF_{+1}$ for $\arg k=\pi/4$. 
\end{itemize}
This defines all four Jost solutions in $\overline{S_\delta}$.
\qed

\subsection{Extension to $k\in\bbC\setminus\{0\}$ via symmetry transformations}
So far, we have proved the existence of Jost solutions for $k\in\overline{S_0}$ for \emph{any potential} $Q$. Here will leverage this flexibility to extend the existence of Jost solutions to all $k\not=0$, by varying the potential. In this subsection it will be convenient to indicate the dependence of the Jost solutions on the potential by writing $\oldF_j(x,k;Q)$.

\emph{Step 1: extension to $-\pi/4\leq \arg k\leq\pi/4$.}
We first observe the following transformation of solutions of the eigenvalue equation \eqref{e30}:
\[
\mbox{if } \oldF(x,k;Q) \mbox{ is a solution of \eqref{e30}, then } \overline{\oldF(x,\overline{k};\overline{Q})} \mbox{ is a solution of \eqref{e30}.}
\]
Using this observation and inspecting the asymptotics of solutions $\oldF_{j}$ in $\overline{S_0}$, we find that formulas
\begin{align*}
\oldF_{\pm\ii}(x,k;Q)&=\overline{\oldF_{\mp\ii}(x,\overline{k};\overline{Q})},
\\
\oldF_{\pm1}(x,k;Q)&=\overline{\oldF_{\pm1}(x,\overline{k};\overline{Q})},
\end{align*}
define the solutions satisfying the hypothesis of Theorem~\ref{thm.e1} for $-\pi/4\leq\arg k\leq 0$. Thus, at this step of the proof we have constructed the Jost solutions for all non-zero $k$ with $-\pi/4\leq\arg k\leq \pi/4$. 

\emph{Step 2: extension to $\pi/4\leq\arg k\leq 3\pi/4$.}
We denote 
\[
\xi=
\begin{pmatrix}
1&0\\0&-1
\end{pmatrix}.
\]
It is important that $\xi e_+=e_-$, $\xi e_-=e_+$ and $\xi\epsilon\xi=-\epsilon$. This step relies on the following transformation:
\[
\mbox{if } \oldF(x,k;Q) \mbox{ is a solution of \eqref{e30}, then } \xi{\oldF(x,-\ii k;-\xi Q\xi}) \mbox{ is a solution of \eqref{e30}.}
\]
Using this and inspecting the asymptotics of the Jost solutions, we find that formulas
\begin{align*}
\oldF_{\pm\ii}(x,k;Q)&=\xi \oldF_{\mp1}(x,-\ii k;-\xi Q\xi),
\\
\oldF_{\pm1}(x,k;Q)&=\xi \oldF_{\pm\ii}(x,-\ii k;-\xi Q\xi),
\end{align*}
define the solutions satisfying the hypothesis of Theorem~\ref{thm.e1} for any non-zero $k$ with $\pi/4\leq\arg k\leq 3\pi/4$.

\emph{Step 3: extension to $3\pi/4\leq\arg k\leq 7\pi/4$.}
Finally, noticing the invariance of the eigenvalue equation~\eqref{e30} under the transformation $k\mapsto-k$, we may use formulas 
\begin{align*}
\oldF_{\pm\ii}(x,k;Q)&=\oldF_{\mp\ii}(x,-k;Q),
\\
\oldF_{\pm1}(x,k;Q)&=\oldF_{\mp1}(x,-k;Q),
\end{align*}
to extend the definition of the Jost solutions to the half-plane $3\pi/4\leq\arg k\leq 7\pi/4$, filling the complete punctured $k$-plane. (This is equivalent to applying the transformation of Step 2 twice.)
The proof of Theorem~\ref{thm.e1} is complete. 
\qed

\subsection{Proof of Proposition~\ref{cr.e3}: reduction to the scalar anti-linear equation}
It is easy to check that if $F$ satisfies the differential equation \eqref{e30}, with $k^2$ real and $Q$ as in~\eqref{eq:a2}, then $\epsilon\overline{F}$ satisfies the same equation. Consider $F=\oldF_{-1}(x,k)$; by inspection, $-\epsilon\overline{F}$ satisfies the same asymptotics as $F$. 
Since the class $\oldF_{-1}$ for $k>0$ consists of a single solution, we conclude that 
\[
-\epsilon\overline{\oldF_{-1}(x,k)}=\oldF_{-1}(x,k),
\]
and so we can represent $\oldF_{-1}(x,k)$ as in \eqref{e50} with some function $\newe$. Rewriting the differential equation and the asymptotics for $\oldF_{-1}$ in terms of $\newe$, we obtain \eqref{e48}, \eqref{e49}.
In a similar way, inspecting $F=\oldF_{+\ii}(x,\ii k)$, we find that 
\[
\epsilon\overline{\oldF_{+\ii}(x,\ii k)}=\oldF_{+\ii}(x,\ii k)
\]
and from here we infer the representation \eqref{e50} for $\oldF_{+\ii}(x,\ii k)$ with some function $\newe$ which satisfies \eqref{e48}, \eqref{e49}.

If $\newe$ satisfies  \eqref{e48} and \eqref{e49}, it is straightforward to see that $\oldF_{-1}$ defined by formula \eqref{e50} satisfies the conditions of Theorem~\ref{thm.e1}.  From here and the uniqueness of $\oldF_{-1}$ we obtain the uniqueness of $\newe$. 

Finally, $\newe$ satisfies \eqref{e6a} if and only if $\oldF_{-1}$ satisfies $L_\alpha^\perp(\oldF_{-1})=0$. 
The proof of Proposition~\ref{cr.e3} is complete. \qed

\subsection{Explicit integral equations for $\myF_{-1}$ and $\myF_{+\ii}$ }
The most important solutions for our analysis are $\myF_{-1}(x,k)$ and $\myF_{+\ii}(x,k)$, see Theorem~\ref{thm.5.3}. We finish this section by displaying explicitly the integral equations for these solutions, with $0\leq \arg k<\pi/4$. 

For $\myF_{-1}$, equation \eqref{eq:t1} reads 
\begin{align}
\myF_{-1}(x,k)
=&\ee^{-kx} e_-
+\frac1k\int_x^\infty \sinh(k(x-y))P_-Q(y)\myF_{-1}(y,k)\dd y
\notag
\\
&-\frac1k\int_x^\infty \sin(k(x-y))P_+Q(y)\myF_{-1}(y,k)\dd y, \quad 0\leq\arg k\leq\pi/4.
\label{eq:e64}
\end{align}
For $\myF_{+\ii}$, \eqref{eq:t2} gives
\begin{align}
\myF_{+\ii}(x,k;r)
=&\ee^{\ii kx}e_+
-\frac1k\int_x^\infty \sin(k(x-y))P_+Q(y)\myF_{+\ii}(y,k;r)\dd y
\notag
\\
&+\frac1{2k}\int_r^\infty \ee^{-k\abs{x-y}}P_-Q(y)\myF_{+\ii}(y,k;r)\dd y, \quad 0\leq\arg k<\pi/4
\label{eq:e65}
\end{align}
for $x>r$, where $r\geq0$ is any parameter satisfying \eqref{e38}. 
We observe that \eqref{eq:e65} is a Fredholm type equation, but not a Volterra type equation.

\section{Jost solutions: analyticity, continuity and large $\abs{k}$ asymptotics}\label{sec:cont}

\subsection{Analyticity and continuity with respect to $k$}
In this section we work in the same framework and under the same assumptions as in the previous section. We discuss the analyticity, continuity, and large $\abs{k}$ asymptotics of the Jost solutions ${\myF}_{j}={\myF}_{j}(x,k;r)$. We recall that ${\myF}_{j}$ are initially defined for $x\geq r$ (as solutions of the integral equations \eqref{eq:t1}, \eqref{eq:t2}), but then extended to $[0,\infty)$ as solutions of the differential equation \eqref{e30}. Below, we write ${\myF}_{j}'=\frac{\partial}{\partial x}{\myF}_{j}$. 

\begin{theorem}\label{thm.e2}
Let $\delta>0$ and $r\geq0$ satisfy \eqref{e38}. Then for any $x\geq0$ and any $j\in\{\pm1,\pm\ii\}$, the Jost solution ${\myF}_{j}(x,k;r)$ and its derivative ${\myF}_{j}'(x,k;r)$ are analytic in $k\in S_\delta$ and continuous in $k\in\overline{S_\delta}$. 
\end{theorem}
\begin{proof} 
First let us assume that $x\geq r$. 
It follows from \eqref{eq:t4}--\eqref{eq:t5} that the functions  $w_{j}$ (see \eqref{eq:t3} for their definition) can be represented as Neumann series
\[
w_{j}=\sum_{m=0}^\infty \calK_{j}^m e_\pm, 
\]
cf. \eqref{eq:ns}. 
The Neumann series converges in $L^\infty((r,\infty);\bbC^2)$ uniformly over $k\in\overline{S_\delta}$. Its $m$th term reads
\[
\calK_{j}^m e_\pm(x)
=
\left(\int_r^\infty\!K_{j}(x-x_1)Q(x_1)\cdots\!\int_r^\infty\!K_{j}(x_{m-1}-x_{m})Q(x_m)
\dd x_1 \cdots \dd x_m\,\right)\!e_\pm.
\]
The integrand is an analytic function of $k$ in $S_\delta$ and bounded by 
\[
\frac{2^m}{\abs{k}^m}\,\norm{Q(x_1)}\cdots\norm{Q(x_m)},
\]
see \eqref{e45}. It follows that $\calK_{j}^m e_\pm(x)$ is analytic in $k\in S_\delta$. By the uniform convergence of the Neumann series, we find that $w_{j}(x)$ is analytic in $k\in S_\delta$. It follows that ${\myF}_{j}(x,k;r)$ is analytic in $k\in S_\delta$. By the same reasoning, one proves the analyticity of ${\myF}_{j}'(x,k;r)$ in $k\in S_\delta$. 

Now consider $0\leq x<r$. Observe that  ${\myF}_{j}(x,k;r)$ on the interval $[0,r]$ can be considered as a solution of the Cauchy problem with the Cauchy data at the point $x=r$. This implies the analytic dependence on $k$ also for $x$ in this range. 

Finally, one can replace analyticity in $S_\delta$ by continuity in $\overline{S_\delta}$ throughout the whole argument above. 
\end{proof}

\subsection{The large $\abs{k}$ asymptotics of $\myF_{-1}$ and $\myF_{+\ii}$}
In the proofs of Theorem~\ref{thm.a1} and Proposition~\ref{prop:M_func_asympt} we will need the $\abs{k}\to\infty$ asymptotics of the Jost solutions $\myF_{-1}(0,k;r)$  and $\myF_{+\ii}(0,k;r)$ in the sector $\overline{S_0}$. Although similar statements can be proved for other Jost solutions, for a more general range of $k$, and for all $x\geq0$, here we establish only the result that we need. In order to simplify the asymptotic expressions below, we only consider $Q$ of the symmetric form \eqref{eq:a2}. We observe that for $\abs{k}$ large, one can always take $r=0$ in the definition of ${\myF}_{j}$, see \eqref{e38}, and we will use this choice in the statement below. 

\begin{lemma}\label{lma.cont2}
Let $q\in L^1(\bbR_+)$ and let $Q$ be of the form \eqref{eq:a2}. Denote 
\[
q_0=\frac12\int_0^\infty \Re q(x)\dd x.
\]
As $\abs{k}\to\infty$ in the sector $0\leq\arg k\leq\pi/4$,  the Jost solutions 
${\myF}_{-1}(x,k)={\myF}_{-1}(x,k;0)$ and ${\myF}_{+\ii}(x,k)={\myF}_{+\ii}(x,k;0)$ satisfy the asymptotics 
\begin{align}
{\myF}_{-1}(0,k)&=e_-+\frac1k q_0 e_-+\smallO(1/k),
\label{eq:F-1-k}
\\
{\myF}_{-1}'(0,k)&=-ke_--q_0 e_-+\smallO(1), 
\label{eq:F-1-prime-k}
\\
{\myF}_{+\ii}(0,k)&=e_++\frac{\ii}{k}q_0 e_++\smallO(1/k),
\label{eq:F-i-k}
\\
{\myF}_{+\ii}'(0,k)&=\ii ke_+-q_0 e_++\smallO(1).
\label{eq:F-i-prime-k}
\end{align}
\end{lemma} 
\begin{proof}
Let us start with \eqref{eq:F-1-k} and \eqref{eq:F-1-prime-k}. The proof follows from a more detailed analysis of the integral equation \eqref{eq:t4} for $w_{-1}$. We revisit \eqref{e57} and see that the norm of the integral operator $\mathcal{K}_{-1}$ is $\calO(1/k)$ as $\abs{k}\to\infty$. It follows that the solution of $w_{-1} = e_- + \mathcal{K}_{-1} w_{-1}$ satisfies the uniform estimate
\[
w_{-1}=e_-+\mathcal{K}_{-1}e_-+\calO(1/k^2)\quad \text{as} \quad \abs{k}\to\infty.
\]
Now let us return to ${\myF}_{-1}$ and rewrite the last relation as
\begin{align*}
{\myF}_{-1}(x,k)=\ee^{-kx}e_- 
&+\frac1k\int_x^\infty\sinh(k(x-y))P_-Q(y)\ee^{-ky}e_- \dd y
\\
&-\frac1k\int_x^\infty \sin(k(x-y))P_+Q(y)\ee^{-ky}e_- \dd y+\calO(1/k^2),
\end{align*}
cf.~\eqref{eq:e64}.
Let us set $x=0$ and use the matrix identities
\begin{equation}
P_+Q(x)e_-=-\ii \Im q(x) e_+, \quad 
P_-Q(x)e_-=-\Re q(x)e_-.
\label{eq:aux_id_1a}
\end{equation}
This gives
\begin{align*}
{\myF}_{-1}(0,k)=e_- 
&+\frac1k\int_0^\infty\sinh(ky)\ee^{-ky}\Re q(y)  \dd y\,  e_-
\\
&-\frac{\ii}{k}\int_0^\infty \sin(ky)\ee^{-ky}\Im q(y) \dd y\,  e_+ +\calO(1/k^2).
\end{align*}
Evaluating the asymptotics of the integrals in the right-hand side here, we arrive at \eqref{eq:F-1-k}. In a similar way we obtain \eqref{eq:F-1-prime-k}.

For \eqref{eq:F-i-k} and \eqref{eq:F-i-prime-k}, the argument is very similar, with the only difference that one starts with the integral equation \eqref{eq:e65}. We skip the details of this calculation. 
\end{proof}

\section{Characterisation of the spectral pair}\label{sec.ee}

The aim of this section is to prove Theorems~\ref{thm.a1} and \ref{thm.5.3}.

\subsection{The matrix solution $E$}
Let us put together two exponentially decaying Jost solutions to form a $2\times 2$ matrix. 
Fix $\delta>0$ and $r\geq0$ such that \eqref{e38} holds.  For $k\in\overline{S_\delta}$, we set 
\begin{equation}
E(x,k)=\{{\myF}_{+\ii}(x,k;r),{\myF}_{-1}(x,k;r)\}.
\label{e8b}
\end{equation}
We suppress the dependence of $E(x,k)$ on $r$ in our notation. In fact, the choice of $r$ will play no role in the argument below and the dependence of relevant quantities on $r$ will eventually disappear. 

Of crucial importance for us will be the dependence of $E(x,k)$ on $k\in\overline{S_\delta}$ as $k$ approaches the positive half-line from above. We recall that by Theorem~\ref{thm.e2}, ${\myF}_{-1}$ and ${\myF}_{+\ii}$ are continuous in $k$ up to the boundary ray $k>0$. Furthermore, by Lemmas~\ref{lma:t5} and \ref{lma:t6}, the solutions ${\myF}_{-1}$ and ${\myF}_{+\ii}$ satisfy the asymptotics \eqref{eq:F-1}--\eqref{eq:F+i-prime} for all $k\in S_0$ as well as for $k>0$.

We also denote
\[
 L_{\alpha}^{\perp}(E(\cdot,k))
=\{L_{\alpha}^{\perp}({\myF}_{+\ii}),L_{\alpha}^{\perp}({\myF}_{-1})\}.
\]
For the purposes of comparing with the scalar self-adjoint case, we note the identity 
\[
\epsilon L_{\alpha}^{\perp}(E(\cdot,k))
=\Phi(0,\overline{k^2})^*\epsilon E'(0,k)-\Phi'(0,\overline{k^2})^*\epsilon E(0,k),
\]
which should be compared to \eqref{eq:sa4a}. 

We remind the reader that the spectrum of $\bH$ is symmetric around $0$, and so it suffices to consider the spectral parameter $\lambda>0$ below. 

\begin{lemma}\label{lma:e3}
Fix $\delta>0$ and $r\geq0$ such that \eqref{e38} holds, and for $k\in\overline{S_\delta}$ let the solution $E$ be defined by \eqref{e8b}. 
\begin{enumerate}[\rm (i)]
\item
For all $k\in\overline{S_\delta}$,  we have $\det  L_{\alpha}^{\perp}(E(\cdot,k))=0$ if and only if $k^2$ is an eigenvalue of $\bH$. 
\item
The determinant $\det L_{\alpha}^{\perp}(E(\cdot,k))$ is analytic in $k\in S_\delta$
and continuous in $k\in\overline{S_\delta}$.  
\end{enumerate}
\end{lemma}

\begin{proof}
\emph{Part (i):} First suppose $\lambda>\delta^2$ is an eigenvalue of $\bH$. Setting $k=\sqrt{\lambda}>0$, we see from Theorem \ref{thm.e1} that out of the four linearly independent Jost solutions $\oldF_{\pm1}$, $\oldF_{\pm\ii}$, only $\oldF_{-1}$ is square-integrable, and so it must satisfy the boundary condition at zero $L_\alpha^\perp(\oldF_{-1})=0$. We recall that $\oldF_{-1}=\myF_{-1}$ for $k>0$. 
From here we find that $\det L_{\alpha}^{\perp}(E(\cdot,k))=0$. 

Conversely, suppose $\det L_{\alpha}^{\perp}(E(\cdot,k))=0$ for some $k\in\overline{S_\delta}$. 
It follows that
\[
a L_\alpha^\perp(\myF_{+\ii})+b L_\alpha^\perp(\myF_{-1})=0
\]
for some $a,b\in\bbC$ not simultaneously equal to zero. We denote 
\[
\widetilde{F}=a\myF_{+\ii}+b\myF_{-1}.
\]
If $\Im k>0$, then $\widetilde{F}$ is square-integrable and since $L_\alpha^\perp(\widetilde{F})=0$, $k^{2}$ would be a~non-real eigenvalue of the self-adjoint operator $\bH$. This is impossible and so $\Im k=0$. Since $k\in\overline{S_{\delta}}$, we conclude that $k>0$. 

Next, we apply a variant of the constancy of Wronskian argument. With the notation \eqref{eq:wr-def}, the Wronskian $[\widetilde{F},\widetilde{F}](x)$ is independent of $x$ (here it is important that $k^2$ is real). Using the asymptotics of $\widetilde{F}$ as $x\to\infty$, see Lemma~\ref{lma:t6}, one computes that 
\[
\lim_{x\to\infty}[\widetilde{F},\widetilde{F}](x)=4\ii k|a|^{2}. 
\]
On the other hand, using that $L_\alpha^\perp(\widetilde{F})=0$, from \eqref{eq:wr4} we find that $[\widetilde{F},\widetilde{F}](0)=0$. Therefore $a=0$ and $\widetilde{F}$ is a non-zero multiple of the exponentially decaying solution $\myF_{-1}$. It follows that $k^{2}$ is an eigenvalue of $\bH$.

\emph{Part (ii):} The claim follows from Theorem~\ref{thm.e2} because $\det L_{\alpha}^{\perp}(E(\cdot,k))$ is a~combination of ${\myF}_{-1}(0,k;r)$, ${\myF}_{+\ii}(0,k;r)$ and the derivatives ${\myF}_{-1}'(0,k;r)$, ${\myF}_{+\ii}'(0,k;r)$.
The proof is complete. 
\end{proof}

\subsection{Representation for the $M$-function in terms of Jost solutions}
Let $\delta>0$ and $r\geq0$ be such that \eqref{e38} holds and let $E$ be defined by \eqref{e8b} for $k\in\overline{S_\delta}$. 

\begin{lemma}\label{lma:e3a}
For any $k\in S_\delta$, the matrix $L_\alpha^\perp(E(\cdot,k))$ is non-singular and 
\begin{align}
M_\alpha(\lambda)=L_\alpha(E(\cdot,k))( L_\alpha^\perp(E(\cdot,k)))^{-1}\epsilon, \quad \lambda=k^2.
\label{e4}
\end{align}
\end{lemma}
\begin{proof}
The matrix $L_\alpha^\perp(E(\cdot,k))$ is non-singular in $S_\delta$ by Lemma~\ref{lma:e3}(i). Consider the solution $X(x,\lambda)$ defined by \eqref{eq:X}. We observe that by the Cauchy data \eqref{eq:phi,theta_bc_cond} for $\Theta$, $\Phi$, we have
\begin{align}
L_\alpha(X)&=L_\alpha(\Theta)-L_\alpha(\Phi)M_\alpha(\lambda)=M_\alpha(\lambda), \label{eq:L_X}
\\
L_\alpha^\perp(X)&=L_\alpha^\perp(\Theta)-L_\alpha^\perp(\Phi)M_\alpha(\lambda)=\epsilon, \label{eq:L_X_perp}
\end{align}
and so we can write
\begin{equation}
M_\alpha(\lambda)=L_\alpha(X)(L_\alpha^\perp(X))^{-1}\epsilon. 
\label{e3}
\end{equation}
Next, from the asymptotics of Theorem~\ref{thm.e1} we see that the solutions $\myF_{-1}$ and $\myF_{+\ii}$ span the two-dimensional linear space of all solutions that belong to $L^2(\bbR_+;\bbC^2)$, see Lemma~\ref{lem:dim_S}. Thus, there is a
non-singular matrix $D$ (which depends on $k$) such that
\[
X(x,k^2)=E(x,k)D.
\]
After substitution into \eqref{e3}, the matrix $D$ cancels out, and we arrive at \eqref{e4}. 
\end{proof}

We note that our definition \eqref{e8b} of $E$ depends on the choice of the parameter~$r$. However, the $M$-function is independent of $r$.

\subsection{Proof of Theorem~\ref{thm.a1}}

\emph{Part (i):}
First let us discuss $\lambda=0$. If $0$ is an eigenvalue of $H$, then, by taking the complex conjugate of the corresponding eigenfunction, we see that $0$ is also the eigenvalue of $H^*$. From the definition of $\bH$ (or from \eqref{eq:e30b})  it follows that $0$ is an eigenvalue of $\bH$ of multiplicity $\geq2$. But the multiplicity of the eigenvalues of $\bH$ cannot be $>2$ because there are only two linearly independent solutions of the eigenvalue equation satisfying the boundary condition at zero. 

Let us prove that all eigenvalues  $\lambda\not=0$ are simple. By \eqref{eq:e30b}, it suffices to consider $\lambda>0$. In this case, setting $k=\sqrt{\lambda}>0$, by the asymptotics of Theorem~\ref{thm.e1}, only one of the four linearly independent solutions of \eqref{e30} belongs to $L^2(\bbR_{+};\bbC^{2})$, and so $\lambda$ can only be a simple eigenvalue. Moreover, $\lambda>0$ is an eigenvalue if and only if $L_\alpha^\perp(F_{-1})=0$. 

Let us check that the set of eigenvalues is bounded. Inspecting the asymptotics \eqref{eq:F-1-k} and \eqref{eq:F-1-prime-k} of ${\myF}_{-1}$ for large $k>0$, we find that condition $L_\alpha^\perp(\myF_{-1})=0$ is incompatible with this asymptotics. 

Next, we prove that the eigenvalues cannot accumulate outside zero by using a standard compactness argument. For a contradiction, suppose that we have a sequence of distinct positive eigenvalues $\lambda_n\to\lambda_*$ with $\lambda_*>0$. Let us denote $k_n=\sqrt{\lambda_n}$,  $k_*=\sqrt{\lambda_*}$ and $g_n(x)=\myF_{-1}(x,k_n)$, $g_*(x)=\myF_{-1}(x,k_*)$.  
Since each $g_n$ satisfies the boundary condition \eqref{eq:a2a0} at zero and $F_{-1}(x,k)$ and $F_{-1}'(x,k)$ are continuous in $k$ by Theorem~\ref{thm.e2}, we find that $g_*$ also satisfies the boundary condition, hence $\lambda_*$ is also an eigenvalue. Let us prove the relation 
\begin{equation}
\lim_{n\to\infty}\int_0^\infty {\jap{g_n(x),g_*(x)}}_{\bbC^2}\,\dd x=\norm{g_*}_{L^2}^2>0;
\label{e51}
\end{equation}
this will give a contradiction with the orthogonality of the eigenvectors $g_n$, $g_*$ for all $n$. 
Recalling our notation from~\eqref{eq:t3}, we have
\[
g_n(x)=\myF_{-1}(x,k_n)=\ee^{-k_n x}\,w_{-1}(x,k_n). 
\]
With this notation, the integral in \eqref{e51} can be written as 
\begin{equation}
\int_0^\infty \ee^{-(k_*+k_n)x}{\jap{w_{-1}(x,k_n),w_{-1}(x,k_*)}}_{\bbC^2}\,\dd x.
\label{e53}
\end{equation}
As established in the proof of Theorem~\ref{thm.e2}, we have 
\[
\lim_{n\to\infty}w_{-1}(x,k_n)=w_{-1}(x,k_*)
\]
for all $x>0$, and the functions $w_{-1}(\cdot,k_n)$ are uniformly bounded. 
By the dominated convergence, the integral \eqref{e53} converges to $\norm{g_*}_{L^2}^2$. Thus, we get \eqref{e51} and the proof of part (i) is complete.

\emph{Parts (ii) and (iii):}
Let $\Delta\subset(0,\infty)$ be a compact interval that contains no eigenvalues of $\bH$. Fix $\delta>0$ such that $\Delta\subset [\delta^2,\infty)$ and let $r\geq0$ satisfy \eqref{e38}. 
Let~$\lambda$ be in the closed strip $\Re \lambda\in\Delta$ and $\Im \lambda\in[0,\eps]$ with $\eps>0$ small and suppose $k^2=\lambda$ with $k\in S_\delta$. 

Consider the expression for $M_\alpha(\lambda)$ in \eqref{e4}. By Theorem~\ref{thm.e2}, the matrices $E(0,k)$ and $E'(0,k)$ are continuous in $k$. By Lemma~\ref{lma:e3}(i), the matrix $L_\alpha^\perp(E(\cdot,k))$ is non-singular in our strip, and therefore its inverse is also continuous. It follows that $M_\alpha(\lambda)$ is continuous in our strip, including the boundary on the real axis. This implies that the measure $\Sigma$ on the interval $\Delta$ contains no singular part. It follows that the singular continuous spectrum of $\bH$ is absent. The remaining property that $\Im M_\alpha(\lambda+\ii0)$ has rank one is proved in the next subsection. The proof of Theorem~\ref{thm.a1} is complete.
\qed

\subsection{Proof of Theorem~\ref{thm.5.3}}

\emph{Part (i):} 
Let us first discuss the determinant $\det\{L_\alpha^\perp(\oldF_{+\ii}),L_\alpha^\perp(\oldF_{-1})\}$ in the right-hand side of \eqref{ee5}. We recall that for $k>0$, the Jost solution $\oldF_{-1}$ is unambiguously defined, while $\oldF_{+\ii}$ is defined up to an additive term $c\oldF_{-1}$, with $c\in\bbC$. This term cancels out in the determinant, and so the determinant is unambiguously defined. 

Next, regarding $\oldF_{+\ii}$ and $\oldF_{-1}$ as solutions classes, we have $\myF_{+\ii}\in\oldF_{+\ii}$ and $\myF_{-1}\in\oldF_{-1}$, and so we can replace $\det\{L_\alpha^\perp(\oldF_{+\ii}),L_\alpha^\perp(\oldF_{-1})\}$ by $\det\{L_\alpha^\perp(\myF_{+\ii}),L_\alpha^\perp(\myF_{-1})\}$ in \eqref{ee5}. Here $\myF_{j}=\myF_{j}(x,k;r)$ for any  $r$ such that \eqref{e38} is satisfied with $\delta=k$.

Let $k>0$ and suppose that $\lambda=k^{2}$ is not an eigenvalue of $\bH$. Passing to the limit in \eqref{e4}, we find
\[
M_\alpha(\lambda+\ii 0)=L_\alpha(E(\cdot,k))( L_\alpha^\perp(E(\cdot,k)))^{-1}\epsilon,
\]
where $E$ is as in \eqref{e8b}.
As explained in the proof of Theorem~\ref{thm.a1} above, we are allowed to pass to the limit by Theorem~\ref{thm.e2}, as $E$ is continuous in $k$ and $L_\alpha^\perp(E(\cdot,k))$ is non-singular (by Lemma~\ref{lma:e3}(i)). Taking the imaginary parts, we find
\[
2\ii \Im M_\alpha(\lambda+\ii 0)
=
\epsilon (B^*)^{-1}(B^*\epsilon C-C^*\epsilon B)B^{-1}\epsilon, 
\]
where we temporarily denote
\begin{align*}
B=L_\alpha^\perp(E(\cdot,k)),
\quad
C=L_\alpha(E(\cdot,k)).
\end{align*}
With a little of elementary matrix algebra (compare with \eqref{eq:wr4}) we compute 
\[
B^*\epsilon C-C^*\epsilon B=E(0,k)^*\epsilon E'(0,k)-E'(0,k)^*\epsilon E(0,k).
\]
The expression in the right-hand side here is a variant of a Wronskian, and since $k^2$ is real, we find that 
\begin{equation}
E(x,k)^*\epsilon E'(x,k)-E'(x,k)^*\epsilon E(x,k)
\label{eq:t18}
\end{equation}
is independent of $x$. We recall that $E$ is defined in terms of the solutions $\myF_{-1}$ and $\myF_{+\ii}$, and these solutions satisfy the large $x$ asymptotics of Theorem~\ref{thm.e1} (see Lemma~\ref{lma:t6}(i)). Computing the limit of the expression \eqref{eq:t18} as $x\to\infty$ by using this asymptotics, we find that it equals $4\ii kP_1$, where $P_1$ is the projection onto $e_1$, i.e. 
\[
P_1=\begin{pmatrix}
    1 & 0 \\ 0 & 0
\end{pmatrix}.
\]
Putting this together gives
\[
\Im M_\alpha(\lambda+\ii0)
=
2k\, 
\epsilon (B^*)^{-1}P_1B^{-1}\epsilon.
\]
Notice that the matrix on the right is of rank one. Denoting the matrix elements
\[
B=\begin{pmatrix}a&b\\c&d\end{pmatrix},
\]
this yields
\[
\Im M_\alpha(\lambda+\ii0)
=
\frac{2k}{\abs{\det B}^2}
\begin{pmatrix}
\abs{b}^2&-\bar{b}d
\\
-b\bar{d}&\abs{d}^2
\end{pmatrix}.
\]
By \eqref{e8b} and \eqref{e50}, 
\[
\begin{pmatrix}
b\\ d
\end{pmatrix}
=L_\alpha^\perp(\myF_{-1})
=L_\alpha^\perp(\oldF_{-1})
=\begin{pmatrix}
\overline{\ell_{\alpha}^\perp(\newe)}\\-\ell_\alpha^\perp(\newe)
\end{pmatrix}.
\]
It follows that 
\[
\Im M_\alpha(\lambda+\ii0)
=
\frac{2k}{\abs{\det\{L_\alpha^\perp(\oldF_{+\ii}),L_\alpha^\perp(\oldF_{-1})\}}^2}
\begin{pmatrix}
\abs{\ell_\alpha^\perp(\newe)}^2& \ell_\alpha^\perp(\newe)^2
\\
\overline{\ell_\alpha^\perp(\newe)}^2& \abs{\ell_\alpha^\perp(\newe)}^2
\end{pmatrix}.
\]
From here we obtain 
\[
\frac{\dd\Sigma}{\dd\lambda}(\lambda)=
\frac{1}{\pi}\frac{2k}{\abs{\det\{L_\alpha^\perp(\oldF_{+\ii}),L_\alpha^\perp(\oldF_{-1})\}}^2}
\begin{pmatrix}
\abs{\ell_\alpha^\perp(\newe)}^2& \ell_\alpha^\perp(\newe)^2
\\
\overline{\ell_\alpha^\perp(\newe)}^2& \abs{\ell_\alpha^\perp(\newe)}^2
\end{pmatrix},
\]
and using \eqref{eq:a1}, we directly deduce the formulas \eqref{ee5}. 

\emph{Part (ii):}
Suppose $k>0$ and $\lambda=k^{2}$ is an eigenvalue of $\bH$. By \cite[Lemma 10.2]{PS3}, we have
\[
\Sigma(\{\lambda\})=\frac{\langle\cdot, L_{\alpha}(\oldF_{-1})\rangle_{\bbC^{2}}}{\|\oldF_{-1}\|_{L^{2}}^{2}}L_{\alpha}(\oldF_{-1}).
\]
Using formula~\eqref{e50} for $\oldF_{-1}$, we find that
\[
\Sigma(\{\lambda\})=\frac{1}{2\|e\|^{2}}\begin{pmatrix}
\abs{\ell_\alpha(\newe)}^2& -\overline{\ell_\alpha(\newe)}^2
\\
-\ell_\alpha(\newe)^2& \abs{\ell_\alpha(\newe)}^2
\end{pmatrix},
\]
and the formulas \eqref{eq:nsa8} immediately follow from \eqref{eq:a1}.
The proof of Theorem~\ref{thm.5.3} is complete. \qed

\subsection{The self-adjoint case revisited}
Let us discuss the case of $q=\overline{q}\in L^1(\bbR_+)$ and $\alpha\in\bbR\cup\{\infty\}$.
In this case, formulas~\eqref{ee5} and~\eqref{eq:nsa8} for the spectral pair simplify. We first note that the scalar Jost solution $\newe$ of Proposition~\ref{cr.e3} is real-valued. 

Let $\lambda>0$. First we discuss $\psi(\lambda)$. Since both $\ell_{\alpha}(\newe)$ in \eqref{eq:nsa8} and $\ell_\alpha^\perp(\newe)$ in \eqref{ee5} are real, we get
\[
\psi(\lambda)=
\begin{cases}
-1, \quad &\mbox{ if } \lambda\in\sigma_{\mathrm p}(\bH), \\
1, \quad &\mbox{ if } \lambda\notin\sigma_{\mathrm p}(\bH).
\end{cases}
\]

For $\lambda>0$ an eigenvalue of $\bH$, the formula for $\nu(\{\lambda\})$ from \eqref{eq:nsa8} remains unchanged. However, 
if $\lambda>0$ is not an eigenvalue of $\bH$, the formula for the density $\dd\nu/\dd\lambda$ from \eqref{ee5} simplifies.
First observe that in the self-adjoint case we can write 
\[
f_{+\ii}(x,k)e_{+}\in\oldF_{+\ii}(x,k), \quad k=\sqrt{\lambda}>0,
\]
where $f_{+\ii}$ is the Jost solution of~\eqref{eq:eveq} with the asymptotics~\eqref{eq:sa4}. With this notation, the formula for $\dd\nu/\dd\lambda$ simplifies to
\[
\frac{\dd\nu}{\dd\lambda}(\lambda)=\frac{k}{2\pi}\frac1{\abs{\ell_\alpha^\perp(f_{+\ii}(\cdot,k))}^2}.
\]
Notice that this is in agreement with the formula for the spectral density~\eqref{eq:sa2} since $\dd\sigma(\lambda)=2\dd\nu(\lambda)$ by \cite[Theorem~1.7]{PS3}.

\section{Scattering theory perspective}\label{sec:wr}
The aim of this section is to prove Theorem~\ref{thm:a4}. 

\subsection{Symmetry relations and basis expansions}
We will be using the Wronskians \eqref{eq:wr-def}. Using \eqref{eq:wr4} and the boundary conditions \eqref{eq:wr5}, for any $F$ and $j=1,2$, we find
\begin{align}
[\Phi_j,F](0)=-\jap{\eps e_j,L_\alpha^\perp(F)}=-\jap{e_{3-j},L_\alpha^\perp(F)}.
\label{eq:wr6}
\end{align}
Using the asymptotics of Theorem~\ref{thm.e1}, let us compute the Wronskians of the Jost solutions $\oldF_{\pm\ii}$, $\oldF_{\pm1}$ at infinity:
\begin{align*}
[\oldF_{-1},\oldF_{-1}]&=[\oldF_{-1},\oldF_{+\ii}]=[\oldF_{-1},\oldF_{-\ii}]=0,
\\
[\oldF_{-1},\oldF_{+1}]&=4k,
\\
[\oldF_{+\ii},\oldF_{+\ii}]&=-[\oldF_{-\ii},\oldF_{-\ii}]=4\ii k,
\\
[\oldF_{+\ii},\oldF_{-\ii}]&=[\oldF_{-\ii},\oldF_{+\ii}]=0.
\end{align*}
It is not possible to compute the rest of the Wronskians unambiguously. 

Using this and \eqref{eq:wr6}, we compute the Wronskians of both sides of \eqref{eq:ge1}, \eqref{eq:ge2} with each of the solutions $\oldF_{-1}$, $\oldF_{+\ii}$:
\begin{alignat}{3}
&[\Phi_2,\oldF_{-1}]:\qquad
&&\jap{e_1,L_\alpha^\perp(\oldF_{-1})}
&&=
4k\gamma_{+1}^2,
\label{eq:wr15}
\\
&[\Phi_1,\oldF_{-1}]:\qquad
&&\jap{e_2,L_\alpha^\perp(\oldF_{-1})}
&&=
4k\gamma_{+1}^1,
\label{eq:wr16}
\\
&[\Phi_2,\oldF_{+\ii}]:\qquad
-&&\jap{e_{1},L_\alpha^\perp(\oldF_{+\ii})}
&&=4\ii k\gamma_{+\ii}^2+\gamma_{+1}^2[\oldF_{+1},\oldF_{+\ii}],
\label{eq:wr25}
\\
&[\Phi_1,\oldF_{+\ii}]:\qquad
-&&\jap{e_{2},L_\alpha^\perp(\oldF_{+\ii})}
&&=4\ii k\gamma_{+\ii}^1+\gamma_{+1}^1[\oldF_{+1},\oldF_{+\ii}].
\label{eq:wr26}
\end{alignat}
Next, we relate $\ell_\alpha^\perp(\newe)$ to the coefficients $\gamma$. By the special structure \eqref{e50} of $\oldF_{-1}$ and the definition \eqref{eq:e30c} of $L_\alpha^\perp$,  we have 
\[
L_\alpha^\perp(\oldF_{-1})
=
\begin{pmatrix}
\overline{\ell_{\alpha}^\perp(\newe)}
\\
-\ell_\alpha^\perp(\newe)
\end{pmatrix},
\]
and so \eqref{eq:wr15} and \eqref{eq:wr16} yield
\begin{align}
4k\gamma_{+1}^1&=\overline{\jap{L_\alpha^\perp(\oldF_{-1}),e_2}}=-\overline{\ell_\alpha^\perp(\newe)},
\label{eq:wr13}
\\
4k\gamma_{+1}^2&=\overline{\jap{L_\alpha^\perp(\oldF_{-1}),e_1}}=\ell_\alpha^\perp(\newe),
\label{eq:wr14}
\end{align}
which yields \eqref{eq:ge3}.

\subsection{Proof of Theorem~\ref{thm:a4}(i):}

Suppose $\lambda>0$ is not an eigenvalue of $\bH$.

Formula \eqref{ee5} contains the determinant $\det\{L_\alpha^\perp(\oldF_{+\ii}),L_\alpha^\perp(\oldF_{-1})\}$. Let us compute it in terms of the coefficients $\gamma$. 
It will be convenient to compute the complex conjugate of this determinant:
\begin{align*}
\overline{\det\{L_\alpha^\perp(\oldF_{+\ii}),L_\alpha^\perp(\oldF_{-1})\}}
=
\jap{e_1,L_\alpha^\perp(\oldF_{+\ii})}\!\jap{e_2,L_\alpha^\perp(\oldF_{-1})}
-
\jap{e_2,L_\alpha^\perp(\oldF_{+\ii})}\!\jap{e_1,L_\alpha^\perp(\oldF_{-1})}.
\end{align*}
When we use equations \eqref{eq:wr15}--\eqref{eq:wr26} to compute the determinant, we see that the terms containing $[\oldF_{+1},\oldF_{+\ii}]$ cancel, and we obtain 
\begin{equation}
\overline{\det\{L_\alpha^\perp(\oldF_{+\ii}),L_\alpha^\perp(\oldF_{-1})\}}
=
16\ii k^2(\gamma_{+1}^2\gamma_{+\ii}^1-\gamma_{+1}^1\gamma_{+\ii}^2). 
\label{eq:wr17}
\end{equation}
In particular, the expression $\gamma_{+1}^2\gamma_{+\ii}^1-\gamma_{+1}^1\gamma_{+\ii}^2$ in the right-hand side is unambiguously defined and non-vanishing by Theorem~\ref{thm.5.3}(i).

Since $\lambda>0$ is not an eigenvalue of $\bH$, we have $\ell_\alpha^\perp(\newe)\not=0$ by Proposition~\ref{cr.e3}, and so $\gamma_{+1}^1$ is non-zero by \eqref{eq:wr13}. Finally,  substituting  \eqref{eq:wr17} and \eqref{eq:wr13} into \eqref{ee5}, we obtain claim~(i) of Theorem~\ref{thm:a4}.

\subsection{Proof of Theorem~\ref{thm:a4}(ii):}

Suppose $\lambda>0$ is an eigenvalue of $\bH$.
Let us rewrite formula \eqref{eq:nsa8} for the spectral pair in terms of the coefficients $\gamma$. 

We recall that $\lambda>0$ is an eigenvalue if and only if $\ell_\alpha^\perp(\newe)=0$. By \eqref{eq:wr13}, \eqref{eq:wr14} we see that in our case
\[
\gamma_{+1}^1=\gamma_{+1}^2=0,
\]
and so the expansions \eqref{eq:ge1}, \eqref{eq:ge2} do not contain the terms with $\oldF_{+1}$ on the right-hand side. It follows that the coefficients $\gamma_{\pm\ii}^1$, $\gamma_{\pm\ii}^2$ are unambiguously defined and the only ambiguity is in the coefficients $\gamma_{-1}^1$, $\gamma_{-1}^2$.

We observe that $\lambda>0$ is an eigenvalue if and only if there is a non-trivial linear combination $a_1\Phi_1+a_2\Phi_2$ that belongs to $L^2(\bbR_+)$. This means that for these $a_1$ and $a_2$, the coefficients in the expansions \eqref{eq:ge1}, \eqref{eq:ge2} satisfy 
\begin{align}
a_1\gamma_{+\ii}^1+a_2\gamma_{+\ii}^2&=0,
\label{eq:wr19}
\\
a_1\gamma_{-\ii}^1+a_2\gamma_{-\ii}^2&=0.
\label{eq:wr20}
\end{align}
\begin{lemma}
Under the above assumptions, all coefficients $a_1$, $a_2$, $\gamma_{\pm\ii}^1$, $\gamma_{\pm\ii}^2$ are non-zero. Moreover, we have 
\begin{equation}
\gamma_{+\ii}^2=\overline{\gamma_{-\ii}^1},
\quad
\gamma_{-\ii}^2=\overline{\gamma_{+\ii}^1}
\label{eq:wr21}
\end{equation}
and 
\begin{equation}
\abs{\gamma_{+\ii}^1}=\abs{\gamma_{-\ii}^1}
=
\abs{\gamma_{+\ii}^2}=\abs{\gamma_{-\ii}^2}. 
\label{eq:wr21a}
\end{equation}
\end{lemma}
\begin{proof}
We recall that if $F$ is a solution of the eigenvalue equation \eqref{e30} with $k^2$ real and $Q$ as in~\eqref{eq:a2}, then $\eps\overline{F}$ is also a solution of the same equation. Applying the transformation $X\mapsto \epsilon\overline{X}$ to equations \eqref{eq:ge1} and \eqref{eq:ge2} and inspecting the boundary conditions, we find that 
\[
\eps\overline{\Phi_2}=\Phi_1.
\]
From here the equations \eqref{eq:wr21} follow. 

In the equation
\[
a_1\Phi_1+a_2\Phi_2=(a_1\gamma_{-1}^1+a_2\gamma_{-1}^2)\oldF_{-1},
\]
let us evaluate $L_\alpha$ of both parts. Using \eqref{eq:wr5} and recalling the structure \eqref{e50} of $\oldF_{-1}$, it yields
\begin{equation}
\begin{pmatrix}a_1\\a_2\end{pmatrix}
=
C\begin{pmatrix}\overline{\ell_\alpha(\newe)}\\-\ell_\alpha(\newe)\end{pmatrix}, 
\quad
C=-(a_1\gamma_{-1}^1+a_2\gamma_{-1}^2).
\label{eq:wr22}
\end{equation}
This means, in particular, that $\abs{a_1}=\abs{a_2}$. Since by assumption $a_1$ and $a_2$ are not simultaneously zero, we get that they are both non-zero. 

Suppose $\gamma_{+\ii}^1=0$. 
By the first equation in \eqref{eq:wr19} we get $\gamma_{+\ii}^2=0$. By \eqref{eq:wr21}, we find $\gamma_{-\ii}^1=\gamma_{-\ii}^2=0$. Coming back to the expansions \eqref{eq:ge1}, \eqref{eq:ge2}, we see that $\Phi_1$ and $\Phi_2$ are collinear, which is incompatible with the boundary conditions~\eqref{eq:wr5} for these solutions. This is a contradiction, so $\gamma_{+\ii}^1\not=0$. The same argument shows that none of the coefficients  $\gamma_{\pm\ii}^1$, $\gamma_{\pm\ii}^2$ vanishes. 
Finally, since $a_1$ and $a_2$ are non-vanishing, from \eqref{eq:wr19},  \eqref{eq:wr20} and \eqref{eq:wr21} we get \eqref{eq:wr21a}. 
\end{proof}

Now we can complete the proof of Theorem~\ref{thm:a4}(ii). 
Let us come back to \eqref{eq:wr22}. Bearing~\eqref{eq:nsa8} in mind, it shows that 
\[
\psi(\lambda)=
-\frac{\overline{\ell_\alpha(\newe)}}{\ell_\alpha(\newe)}=\frac{a_1}{a_2}. 
\]
From here and  \eqref{eq:wr19}, \eqref{eq:wr20} we find 
\begin{equation}
\psi(\lambda)
=\frac{a_1}{a_2}
=-\frac{\gamma_{+\ii}^2}{\gamma_{+\ii}^1}
=-\frac{\gamma_{-\ii}^2}{\gamma_{-\ii}^1}.
\label{eq:wr23}
\end{equation}

Next, let us express the numerator $\abs{\ell_\alpha(\newe)}^2$ in the first formula in \eqref{eq:nsa8} in terms of the coefficients $\gamma$. Consider the first equation in \eqref{eq:wr22}:
\[
a_1=-(a_1\gamma_{-1}^1+a_2\gamma_{-1}^2)\overline{\ell_\alpha(\newe)}.
\]
Dividing by $a_1$ and using \eqref{eq:wr23} yields
\[
1=-(\gamma_{-1}^1-\frac{\gamma_{-\ii}^1}{\gamma_{-\ii}^2}\gamma_{-1}^2)\overline{ \ell_\alpha(\newe)}.
\]
It follows that $\gamma_{-\ii}^1\gamma_{-1}^2-\gamma_{-1}^1\gamma_{-\ii}^2\neq0$ and
\begin{equation}
\overline{\ell_\alpha(\newe)}
=\frac{\gamma_{-\ii}^2}{\gamma_{-\ii}^1\gamma_{-1}^2-\gamma_{-1}^1\gamma_{-\ii}^2}.
\label{eq:wr24}
\end{equation}
Plugging this into the first formula from~\eqref{eq:nsa8} and using \eqref{eq:wr21a}, we obtain
\[
\nu(\{\lambda\})=
\frac{1}{2\norm{\newe}^2}
\frac{\abs{\gamma_{+\ii}^1}^2}{\abs{\gamma_{-1}^1\gamma_{-\ii}^2-\gamma_{-\ii}^1\gamma_{-1}^2}^2}.
\]
Moreover, the equation \eqref{eq:wr24} also proves that the denominator $\gamma_{-1}^1\gamma_{-\ii}^2-\gamma_{-\ii}^1\gamma_{-1}^2$
is unambiguously defined, because both $\gamma_{-\ii}^2$ and $\ell_\alpha(\newe)$ are unambiguously defined. 
The proof of Theorem~\ref{thm:a4} is complete.

\section{The Born approximation for the spectral pair}\label{sec:Born}

Here we prove Proposition~\ref{prop:Born}. In fact, we give a more general set of formulas for any boundary parameter $\alpha\in\bbR\cup\{\infty\}$. Below for readability we write $\calO(q^2)$ in place of $\calO(\norm{q}_{L^1}^2)$. 

\subsection{The self-adjoint case}

As a motivation and for comparison, we first briefly deduce the Born approximations for the spectral measure on $\bbR_{+}$ of self-adjoint $H$ with $q=\overline{q}\in L^{1}(\bbR_{+})$ and $\alpha\in\bbR\cup\{\infty\}$. From the integral equation for the Jost solution $f_{+\ii}$, 
\[
f_{+\ii}(x,k)=\ee^{\ii kx}-\frac1k\int_x^\infty q(y)\sin(k(x-y))f_{+\ii}(y,k) \dd y,
\]
we find that
\[
f_{+\ii}(x,k)=\ee^{\ii kx}-\frac1k\int_x^\infty q(y)\sin(k(x-y)) \ee^{\ii ky} \dd y+\calO(q^2), 
\]
and so 
\begin{align*}
f_{+\ii}(0,k)&=1+\frac1k\int_{0}^{\infty}\ee^{\ii ky}\sin (ky)q(y)\dd y+\calO(q^2), \\
f_{+\ii}'(0,k)&=\ii k-\int_{0}^{\infty}\ee^{\ii ky}\cos(ky)q(y)\dd y+\calO(q^{2}).
\end{align*}
Using definition~\eqref{eq:sa6}, we find for $\alpha\in\bbR$, that
\[
\ell_{\alpha}^{\perp}(f_{+\ii})=\frac{1}{\sqrt{1+\alpha^{2}}}\left(\alpha+\ii k+\frac{1}{k}\int_{0}^{\infty}\ee^{\ii ky}\left[\alpha\sin(ky)-k\cos(ky)\right]q(y)\dd y\right)+\calO(q^{2}).
\]
This yields, after a short computation, the expansion
\begin{align*}
|\ell_{\alpha}^{\perp}(f_{+\ii})|^{2}=\frac{1}{1+\alpha^{2}}\bigg(\alpha^{2}+k^{2}&+\frac{\alpha^{2}-k^{2}}{k}\int_{0}^{\infty}\sin(2ky)q(y)\dd y \\
&\hskip26pt-2\alpha\int_{0}^{\infty}\cos(2ky)q(y)\dd y\bigg)+\calO(q^{2}).
\end{align*}
Substituting into \eqref{eq:sa2}, we obtain the final formula
\begin{align}
 \frac{\dd\sigma}{\dd\lambda}(\lambda)=\frac{1}{\pi}\frac{1+\alpha^{2}}{k^{2}+\alpha^{2}} \bigg(k&-\frac{\alpha^{2}-k^{2}}{\alpha^{2}+k^{2}}\int_{0}^{\infty}\sin(2ky)q(y)\dd y \nonumber\\
 &+\frac{2k\alpha}{\alpha^{2}+k^{2}}\int_{0}^{\infty}\cos(2ky)q(y)\dd y\bigg)+\calO(q^{2})
 \label{eq:born_sigma}
\end{align}
for $k=\sqrt{\lambda}>0$ and $\alpha\in\bbR$. When we formally pass to the limit $\alpha\to\infty$ in the above formula, we recover the correct result \eqref{eq:ba1} for the Dirichlet case $\alpha=\infty$.

\subsection{The non-self-adjoint case}
Next we compute expansions analogous to~\eqref{eq:born_sigma} for the spectral pair $(\nu,\psi)$ for complex-valued $q\in L^{1}(\bbR)$ and $\alpha\in\bbR\cup\{\infty\}$.
Expansions for complex $\alpha$ can be calculated by the same computational procedure but the resulting formulas are more complicated. The calculation below is in many aspects analogous to the one in the proof of Lemma~\ref{lma.cont2}.

We assume $Q$ of the form~\eqref{eq:a2} and denote 
\[
u(x)=\Re q(x) \quad\mbox{ and }\quad v(x)=\Im q(x)
\]
for brevity. The integral equation \eqref{eq:e64} for $F_{-1}$ gives
\begin{align*}
F_{-1}(x,k)=\ee^{-kx}e_{-}&+\frac1k\int_x^\infty \ee^{-ky}\sinh(k(x-y))P_-Q(y)e_-\dd y
\\
&-\frac1k\int_x^\infty \ee^{-ky}\sin(k(x-y))P_+Q(y)e_-\dd y
+\calO(q^2)
\end{align*}
for $k>0$. Using the identities from~\eqref{eq:aux_id_1a}, we readily deduce
\begin{align}
F_{-1}(0,k)
=e_-&+\frac1k\int_{0}^{\infty}\ee^{-ky}\sinh(ky)u(y)\dd y\, e_{-} \nonumber\\
&-\frac{\ii}{k}\int_{0}^{\infty} \ee^{-ky}\sin(ky) v(y)\dd y\, e_{+}+\calO(q^2)
\label{eq:ss1}
\end{align}
and
\begin{align}
F_{-1}'(0,k)
=-ke_{-}&-\int_0^\infty  \ee^{-ky}\cosh(ky)u(y)\dd y\,e_- \nonumber\\
&+\ii\int_0^\infty  \ee^{-ky}\cos(ky) v(y)\dd y\,e_{+}+\calO(q^2).
\label{eq:ss1_der}
\end{align}

Similarly, starting from the integral equation~\eqref{eq:e65} for $F_{+\ii}$ (for $q$ small, we can put $r=0$) and using identities
\[
P_+Q(x)e_+=u(x) e_+, \quad 
P_-Q(x)e_+=\ii v(x)e_-, 
\]
results, for $k>0$, in equalities
\begin{align}
F_{+\ii}(0,k)&=e_{+}+\frac{1}{k}\int_0^\infty \ee^{\ii ky}\sin(ky) u(y)\dd y\,e_{+} 
+\frac{\ii}{2k}\int_0^\infty \ee^{\ii ky-ky}v(y)\dd y\, e_{-} +\calO(q^2), \label{eq:ss2} \\
F_{+\ii}'(0,k)&=\ii ke_{+}-\int_0^\infty  \ee^{\ii ky}\cos(ky)u(y)\dd y\,e_{+}
-\frac{\ii}{2}\int_0^\infty  \ee^{\ii ky-ky} v(y)\dd y\,e_{-}+\calO(q^2). \label{eq:ss2_der}
\end{align}

Recalling \eqref{e50}, we find from \eqref{eq:ss1} and \eqref{eq:ss1_der} that 
\begin{align*}
\newe(0,k)&=1+\frac{1}{k}\int_0^\infty \ee^{-ky}\sinh(ky)u(y)\dd y
+\frac{\ii}k\int_0^\infty  \ee^{-ky}\sin(ky) v(y)\dd y+\calO(q^2), \\
\newe'(0,k)&=-k-\int_0^\infty  \ee^{-ky}\cosh(ky)u(y)\dd y-\ii\int_0^\infty  \ee^{-ky}\cos(ky) v(y)\dd y+\calO(q^2).
\end{align*}
Then, for $\alpha\in\bbR$, we obtain
\begin{align}
    \ell_{\alpha}^{\perp}(\newe)&=\frac{1}{\sqrt{1+\alpha^{2}}}\bigg(\alpha-k+\frac{1}{k}\int_{0}^{\infty}\ee^{-ky}\left[\alpha\sinh(ky)-k\cosh(ky)\right]u(y)\dd y \nonumber\\
    &\hskip74pt+\frac{\ii}{k}\int_{0}^{\infty}\ee^{-ky}\left[\alpha\sin(ky)-k\cos(ky)\right]v(y)\dd y\bigg)+\calO(q^2). \label{eq:ell_alp_q_expand}
\end{align}
From here and \eqref{ee5} we deduce the formula
\begin{equation}
 \psi(\lambda)=1+\frac{2\ii}{k(\alpha-k)}\int_{0}^{\infty}\ee^{-ky}\left[\alpha\sin(ky)-k\cos(ky)\right]v(y)\dd y +\calO(q^{2}),
\label{eq:born_psi}
\end{equation}
which holds for $\lambda=\sqrt{k}>0$ and $k\neq\alpha$, if $\alpha>0$.
It is also straightforward to compute from~\eqref{eq:ell_alp_q_expand} that
\[
|\ell_{\alpha}^{\perp}(\newe)|^{2}=\frac{\alpha-k}{1+\alpha^{2}}\bigg(\alpha-k+\frac{2}{k}\int_{0}^{\infty}\ee^{-ky}\left[\alpha \sinh(ky) 
    -k \cosh(ky)\right]u(y)\dd y\bigg)+\calO(q^{2}).
\]
On the other hand, the computation of $\det\{L_\alpha^\perp(\oldF_{+\ii}),L_\alpha^\perp(\oldF_{-1})\}$ is a more laborious application of formulas \eqref{eq:ss1}, \eqref{eq:ss1_der} and \eqref{eq:ss2}, \eqref{eq:ss2_der}. 
Omitting the computational details and using \eqref{ee5}, we obtain the final formula for the density of $\nu$:
\begin{align}
 \frac{\dd\nu}{\dd\lambda}(\lambda)=\frac{1}{2\pi}\frac{1+\alpha^{2}}{k^{2}+\alpha^{2}} \bigg(k&-\frac{\alpha^{2}-k^{2}}{\alpha^{2}+k^{2}}\int_{0}^{\infty}\sin(2ky)u(y)\dd y \nonumber\\
 &+\frac{2k\alpha}{\alpha^{2}+k^{2}}\int_{0}^{\infty}\cos(2ky)u(y)\dd y\bigg)+\calO(q^{2})
 \label{eq:born_nu}
\end{align}
for $k=\sqrt{\lambda}>0$ and $\alpha\in\bbR$.

In the Dirichlet case $\alpha=\infty$, the Born approximation formulas can be obtained by passing to the limit $\alpha\to\infty$ in~\eqref{eq:born_psi} and \eqref{eq:born_nu}, which gives \eqref{eq:ba2} and \eqref{eq:ba3}.

\appendix

\section{The Schr{\" o}dinger operator $H$ with $q\in L^1(\bbR_+)$}
\label{sec:B}

\subsection{Preliminaries}
Here we recall relevant facts of the spectral theory of Schr{\" o}dinger operators on the half-line with complex potentials and indicate the proof of Proposition~\ref{prp:B1}. Our main source is \cite{DG}; we will comment on the history at the end of this appendix.

We denote by $\AC(\bbR_+)$ the set of absolutely continuous functions on $\bbR_+$ and by $\AC^1(\bbR_+)$ the set of $f\in\AC(\bbR_{+})$ such that $f'\in\AC(\bbR_+)$. For $q\in L^1(\bbR_+)$, we set 
\[
D_{\max}(q)=\{f\in L^2(\bbR_+)\cap\AC^1(\bbR_+):  -f''+qf\in L^2(\bbR_+)\}.
\]
We note that for complex-valued $q$, in general $D_{\max} (q)\not=D_{\max} (\overline{q})$; see the striking example of \cite[Lemma~4.10(2)]{DG}.

By \cite[Proposition~5.1]{DG}, any $f\in D_{\max}(q)$ extends to a $C^1$-smooth function on $[0,\infty)$ and in particular the linear functionals $f(0)$, $f'(0)$ and $\ell^\perp_\alpha(f)$ are well-defined as the corresponding limit values. Thus, the definition \eqref{eq:domainH} of the domain of $H=H(q,\alpha)$, which we write as 
\[
\Dom H(q,\alpha)=\{f\in D_{\max}(q): \ell^\perp_\alpha(f)=0\},
\]
makes sense. 

\subsection{The maximal operator and Green's identity}
Along with $H(q,\alpha)$, we consider the maximal operator $H_{\max}=H_{\max}(q)$, defined by the same differential expression as $H$ with the domain $\Dom H_{\max} (q)=D_{\max}(q)$. By \cite[Theorem 4.4(1)]{DG}, the operator $H_{\max} (q)$ is closed and densely defined.

Assuming $q\in L^{1}(\bbR_{+})$, for any $f\in \Dom H_{\max}(q)$ and $g\in \Dom H_{\max}(\overline{q})$ we have Green's identity \cite[Theorem~4.4(3)]{DG}:
\begin{equation}
\jap{H_{\max}(q)f,g}-\jap{f,H_{\max}(\overline{q})g}
=
f'(0)\overline{g(0)}-f(0)\overline{g'(0)}.
\label{eq:B1}
\end{equation}
Here it is critical that the boundary terms at infinity vanish, which is a consequence of the assumption $q\in L^1(\bbR_+)$; in fact, it would suffice to assume 
\[
\limsup_{c\to\infty}\int_{c}^{c+1}\abs{q(x)}\dd x<\infty,
\]
see \cite[Proposition~5.15]{DG}. It is precisely the vanishing of the boundary terms at infinity in \eqref{eq:B1} that is the non-self-adjoint analogue of the limit-point condition.

\begin{lemma}\label{lma:B2}
For any complex numbers $p_0$ and $p_1$, there exists $f\in D_{\max}(q)$ satisfying $f(0)=p_0$ and $f'(0)=p_1$ and compactly supported in $[0,\infty)$. 
\end{lemma}
\begin{proof}
Let $p_0$ and $p_1$ be given. By \cite[Proposition~2.5]{DG}, there exists $g\in\AC^1([0,\infty))$ with $g(0)=p_0$, $g'(0)=p_1$ and $-g''+qg\in L^2(\bbR_+)$. Set $f=g\varphi$, where $\varphi$ is $C^2$-smooth, $\varphi(x)=1$ for $x\leq1$ and $\varphi(x)=0$ for $x\geq2$. Then $f\in \AC^1([0,\infty))\cap L^2(\bbR_+)$, $f$ satisfies the needed boundary conditions at zero and  
\[
-f''+qf=(-g''+qg)\varphi-2g'\varphi'-g\varphi''\in L^2(\bbR_+). 
\]
Thus, $f\in D_{\max}(q)$, as claimed. 
\end{proof}

\subsection{The minimal operator}
Using the notation of \cite[Definition~4.3]{DG}, we denote by $H_c(q)$ the restriction of $H_{\max}(q)$ onto the domain 
\[
D_c(q)=\{f\in D_{\max}(q): f \text{ is compactly supported in }(0,\infty)\}. 
\]
Since $H_{c}(q)\subset H_{\max}(q)$ and $H_{\max}(q)$ is closed, $H_{c}(q)$ is closable and we denote by $H_{\min}(q)$ its operator closure. According to parts (1), (2) and (5) of \cite[Theorem~4.4)]{DG}, the operator $H_{\min}(q)$ is densely defined, $H_{\min}(q)^*=H_{\max}(\overline{q})$, $H_{\max}(q)^*=H_{\min}(\overline{q})$ and 
\begin{equation}
\Dom H_{\min}(q)=\{f\in D_{\max}(q): f(0)=f'(0)=0\}.
\label{eq:B2}
\end{equation}
The fact that $H_{\min}(q)$ is the closure of $H_c(q)$ can be rephrased as follows: the set $D_c(q)$ is dense in $\Dom H_{\min}(q)$ in the graph norm of $H$.

\subsection{Proof of Proposition~\ref{prp:B1}}
In \cite[Theorem~4.4(1)]{DG}, it is proved that $\Dom H_{\min}(q)$ is dense in $L^2(\bbR_+)$. Hence, $\Dom H(q,\alpha)$ is also dense in $L^2(\bbR_+)$. The closedness of $H(q,\alpha)$ will follow from $H(q,\alpha)^*=H(\overline{q},\overline{\alpha})$, rewriting it as $H(q,\alpha)=H(\overline{q},\overline{\alpha})^*$. 

Let us prove that $H(q,\alpha)^*=H(\overline{q},\overline{\alpha})$. Let $f\in\Dom H(q,\alpha)$, $g\in \Dom H(\overline{q},\overline{\alpha})$. Then $\ell_\alpha^\perp(f)=\ell_{\overline{\alpha}}^\perp(g)=0$ and therefore the boundary terms in \eqref{eq:B1} vanish: 
\[
\jap{H(q,\alpha)f,g}=\jap{f,H(\overline{q},\overline{\alpha})g}.
\]
This proves that $H(\overline{q},\overline{\alpha})\subset H(q,\alpha)^*$. 

Let us prove the converse inclusion $H(q,\alpha)^*\subset H(\overline{q},\overline{\alpha})$. 
Note that $H_{\min}(q)\subset H(q,\alpha)$, hence $H(q,\alpha)^*\subset H_{\min}(q)^*=H_{\max}(\overline{q})$.  It remains to prove that any $g\in\Dom H(q,\alpha)^*$ satisfies the boundary condition $\ell_{\overline{\alpha}}^\perp(g)=0$. Using Lemma~\ref{lma:B2}, let us take any $f\in\Dom H_{\max}(q)$ such that
\begin{align*}
f(0)&=1,\quad f'(0)=-\alpha,\quad  \text{ if $\alpha\in\bbC$},
\\
f(0)&=0,\quad f'(0)=1,  \quad\quad\text{ if $\alpha=\infty$}. 
\end{align*}
Then $f\in \Dom H(q,\alpha)$ and so we have 
\[
\jap{H_{\max}(q)f,g}
=\jap{H(q,\alpha)f,g}
=\jap{f,H(q,\alpha)^*g}
=\jap{f,H_{\max}(\overline{q})g}.
\]
By \eqref{eq:B1}, it follows that 
\[
f'(0)\overline{g(0)}=f(0)\overline{g'(0)}.
\]
Substituting the values of $f(0)$ and $f'(0)$, we find that $\ell_{\overline{\alpha}}^\perp(g)=0$, as required. 
 \qed

\subsection{A dense set in the domain of $H(q,\alpha)$}
Here we give a lemma that will be needed in Appendix~\ref{sec:C}. 
\begin{lemma}\label{prp:B1a}
The set 
\[
\{f\in \Dom H(q,\alpha): f \text{ is compactly supported in $[0,\infty)$} \}
\]
is dense in $\Dom H(q,\alpha)$ in the graph norm of $H(q,\alpha)$. 
\end{lemma}
\begin{proof}
Let $f\in\Dom H(q,\alpha)$. Using Lemma~\ref{lma:B2}, take $g\in D_{\max}(q)$ compactly supported in $[0,\infty)$ and satisfying $g(0)=f(0)$, $g'(0)=f'(0)$. Then $h=f-g$ belongs to $H_{\min}(q)$, see \eqref{eq:B2}. Hence, by the definition of $H_{\min}(q)$, there exists a sequence of elements $h_n\in D_c(q)$ with $h_n\to h$ in the graph norm of $H$ as $n\to\infty$. Denote $f_n=h_n+g$; then $f_n\to f$ in the graph norm of $H$. Furthermore, each $f_n$ is compactly supported in $[0,\infty)$, belongs to $D_{\max}(q)$ and satisfies the same boundary condition at zero as $g$, i.e. $\ell_\alpha^\perp(f_n)=0$. Thus, $f_n\in\Dom H(q,\alpha)$. 
\end{proof}

\subsection{Concluding remarks}
Of course, the subject of this appendix is classical and goes back to Weyl and Titchmarsh, who worked with real continuous $q$. Even in the case of real $q$, the description of realisations of $H$ as a closed operator in $L^2$ under the sole assumption $q\in L_{\text{loc}}^1$ has significant technical difficulties because in general there is no known convenient dense set in the domain of $H$.  As far as we are aware, the technical difficulties associated with real $q\in L_{\text{loc}}^1$ were first resolved by Stone \cite{Stone}; see also Krein \cite{Krein} in the finite interval case and Glazman \cite{Glazman} in the half-line case. The extension of the Weyl limit-point/limit-circle analysis to complex potentials goes back to Sims \cite{Sims} and Everitt \cite{Everitt}.

We use a more modern (and very systematic) paper Derezinski--Georgescu \cite{DG}, which is a synthesis of earlier work on this subject. It is specifically focused on describing closed extensions of $H_{\min}(q)$ for complex $q$ under the sole assumption $q\in L_{\text{loc}}^1(\bbR_+)$ (\cite{DG} also consider the full-line and the finite-interval cases). We note that \cite{DG} is presented in the language of \emph{transposed operators} instead of adjoints, but it is easy to translate their statements into the language of adjoints.

\section{The hermitised operator $\bH$ with $q\in L^1(\bbR_+)$}
\label{sec:C}

Here we briefly comment on proofs of Propositions~\ref{prp.Mf} and \ref{prp.sm}. 

\subsection{The dimension lemma}

\begin{lemma}[The dimension lemma]\label{lem:dim_S}
For any $\lambda\in\bbC\setminus\bbR$, the dimension of the space
\begin{equation}
S(\lambda)=\{F\in L^{2}(\bbR_{+};\bbC^{2}) : F \mbox{ is a solution of~\eqref{eq:init_diff_eq}}\}
\label{eq.dim}
\end{equation}
equals $2$.
\end{lemma}

\begin{proof}
The dimension of $S(\lambda)$ for $\lambda$ in the upper and lower half-planes coincides with the deficiency indices of the operator $\bH$ restricted to $\Dom H_{\min}(\overline{q})\oplus\Dom H_{\min}(q)$. In particular, $\dim S(\lambda)$ is constant in the upper and lower half-planes. Take, for example, $\lambda$ in the open first quadrant. Then by Theorem~\ref{thm.e1} there are exactly two linearly independent $L^2$-solutions of the equation \eqref{eq:init_diff_eq}, and so the dimension of $S(\lambda)$ equals 2. The same argument applies for the lower half-plane. 
\end{proof}

\subsection{Proof of Proposition~\ref{prp.Mf}}
The most difficult part of the proof of Proposition~\ref{prp.Mf} is the existence and uniqueness of the matrix $M_\alpha(\lambda)$; it is usually proved in the framework of Weyl's limit-point theory, see e.g. \cite[Appendix~A.2]{PS3} for a self-contained presentation. 

We would like to explain that in our setting $q\in L^1(\bbR_+)$, the existence and uniqueness of the matrix $M_\alpha(\lambda)$ follow easily from the dimension lemma. Fix $\lambda\in\bbC\setminus\bbR$ and let $S(\lambda)$ be as in~\eqref{eq.dim}. Let $\Phi_1$ and $\Phi_2$ be the columns of the solution $\Phi$ and similarly let $\Theta_1$ and $\Theta_2$ be the columns of $\Theta$. We note that $\Phi_1$ and $\Phi_2$ are linearly independent and moreover the linear span of $\Phi_1$ and $\Phi_2$ has zero intersection with $S(\lambda)$, as non-real $\lambda$ cannot be an eigenvalue of $\bH$. It follows that any solution of the eigenvalue equation \eqref{eq:init_diff_eq} can be represented in a unique way as a sum of a solution from $S(\lambda)$ and a solution from the linear span of $\Phi_1$ and $\Phi_2$. Thus, for each non-real $\lambda$ there exists a unique set of coefficients $m_{jk}$ such that
\[
\Theta_1+m_{11}\Phi_1+m_{21}\Phi_2\in S(\lambda)
\quad \text{ and } \quad
\Theta_2+m_{21}\Phi_1+m_{22}\Phi_2\in S(\lambda).
\]
This gives the matrix $M_\alpha(\lambda)$. 

Let us prove the analyticity of $M_\alpha(k^2)$ for $k$ in the first quadrant. The analyticity in $k$ in the  sector $S_0$ and continuity up to the line $\arg k=\pi/4$ follows by the representation of Lemma~\ref{lma:e3a} and by the analyticity and continuity properties of solutions $\myF_{-1}$ and $\myF_{+\ii}$ proved in Theorem~\ref{thm.e2}. In a completely analogous way, one proves the analyticity of $M_\alpha(k^2)$ in the sector $\pi/4<\arg k<\pi/2$ and continuity up to the line $\arg  k=\pi/4$. It is important that the limiting values of $M_\alpha$ on the diagonal $\arg k=\pi/4$ agree, as follows from the uniqueness of the definition of the $M$-function. By a standard argument involving Morera's theorem, we find that $M_\alpha(k^2)$ is analytic in the open first quadrant. 

To complete the proof of Proposition~\ref{prp.Mf}, it remains to prove the property \eqref{eq:a4a}. 
This property follows from the identity 
\begin{equation}
\Im M_{\alpha}(\lambda)=\Im\lambda\int_{0}^{\infty}X(x,\lambda)^{*}X(x,\lambda)\dd x,
\label{eq:id}
\end{equation}
where $\lambda\in\bbC\setminus\bbR$ and $X(\cdot,\lambda)$ is the solution defined by \eqref{eq:X}. To derive \eqref{eq:id}, it suffices to check the identity
\[
\bigl(X'(x,\lambda)^{*}\epsilon X(x,\lambda)-X(x,\lambda)^{*}\epsilon X'(x,\lambda)\bigr)'
=
(\lambda-\overline{\lambda})X(x,\lambda)^{*}X(x,\lambda),
\]
integrate it over $x$ from $0$ to $\infty$ and take into account the boundary conditions \eqref{eq:L_X} and \eqref{eq:L_X_perp} at $x=0$.

\subsection{Proof of Proposition~\ref{prp.sm}(i)}
Let $L^{2}_{\Sigma}(\bbR;\bbC^{2})$ be the space of all $\bbC^2$-valued functions on $\bbR$ equipped with the inner product
\[
\jap{F,G}=\int_{-\infty}^\infty \jap{\dd\Sigma(\lambda)F(\lambda),G(\lambda)}_{\bbC^2}.
\]
Define the operator $U: L^2(\bbR_+;\bbC^2)\to L_{\Sigma}^{2}(\bbR;\bbC^{2})$ by
\[
(UF)(\lambda):=\int_0^\infty \Phi(x,\lambda)^{*}F(x)\dd x
\]
initially on the set of smooth compactly supported functions $F$,  and let $\Lambda$ be the operator of multiplication by the independent variable in $L_{\Sigma}^{2}(\bbR;\bbC^{2})$. Proposition~\ref{prp.sm}(i) can be stated more precisely as follows. 

\begin{proposition}
The operator $U$ extends to a unitary map from $L^2(\bbR_+;\bbC^2)$ onto $L_{\Sigma}^{2}(\bbR;\bbC^{2})$. 
Moreover, $U$ intertwines $\bH$ with $\Lambda$, i.e. $U\Dom \bH\subset \Dom\Lambda$ and 
\begin{equation}
U\bH F=\Lambda UF
\label{eq:intertwine}
\end{equation}
for all $F\in\Dom \bH$.
\end{proposition}

This proposition should be regarded as known; it can be reduced to the theory of Hamiltonian systems developed in great generality by Hinton and Shaw \cite{HS}. However, we found that extracting relevant facts from the literature and translating them into the language of the operator $\bH$ is not an easy task, and therefore in \cite[Appendix~A.5]{PS3} we gave a self-contained proof for bounded $q$. Only Step 2 (the proof of the intertwining relation \eqref{eq:intertwine}) of this proof relied on the boundedness of $q$. Let us explain where the difficulty is and how to overcome it for $q\in L^1(\bbR_+)$. At its core, the proof is a simple integration by parts:
\begin{align*}
U{\bH}F(\lambda)
&=
\int_0^\infty \Phi(x,\lambda)^*(\bH F)(x)\dd x
=
\int_0^\infty (\bH \Phi(x,\lambda))^* F(x)\dd x
\\
&=
\lambda \int_0^\infty \Phi(x,\lambda)^* F(x)\dd x
=
\lambda UF(\lambda)
\end{align*}
on a suitable set of functions $F$. The difficulty is in choosing an appropriate set of functions $F$ such that on one hand, the integration by parts can be justified and on the other hand, this set of functions is dense in $\Dom\bH$ in the graph norm of $\bH$. For bounded $q$, it suffices to take smooth functions $F$ vanishing for all sufficiently large $x$ and satisfying the boundary condition for $\bH$ at the origin. For $q\in L^1(\bbR_+)$, the above choice doesn't necessarily work. However, by Lemma~\ref{prp:B1a}, one can instead take the set of all $F\in\Dom\bH$ compactly supported in $[0,\infty)$, which resolves the difficulty.

\subsection{Proof of Proposition~\ref{prp.sm}(ii)}
Part (ii) of Proposition~\ref{prp.sm} is Theorem~1.2 of \cite{PS3}. This theorem was stated in \cite{PS3} for $q\in L^\infty(\bbR_+)$ but in fact the boundedness of $q$ was not used in their proofs. Indeed, Theorem~1.2 of \cite{PS3} follows from the symmetries of the matrix $Q$ in \eqref{eq:a2}.

\subsection{Proof of Proposition~\ref{prp.sm}(iii)}
Part (iii) of Proposition~\ref{prp.sm} is \cite[Theorem~1.1]{PS3}. One step of the proof given in \cite{PS3} relies on the boundedness of $q$ (see Lemma~A.5 of \cite{PS3}). Below we give a modified proof of \cite[Theorem~1.1]{PS3} that works for $q\in L^1(\bbR_+)$ and relies on our construction of Jost solutions. 

First we give a modified version of \cite[Proposition~3.3]{PS3}. 

\begin{proposition}\label{prop:M_func_asympt}
Assume $q\in L^1(\bbR_+)$. 
Let $\alpha\in\bbC\cup\{\infty\}$ and $\lambda=k^2$ with $0\leq\arg k\leq\pi/4$. 
Then as $\abs{k}\to\infty$ along any ray, the $M$-function satisfies
\begin{equation}
M_\alpha(\lambda)=
\epsilon A+\dfrac{1+|\alpha|^{2}}{k}(\ii P_{+}-P_{-})+\dfrac{1+\abs{\alpha}^2}{k^2}(\ii P_{+}-P_{-})\epsilon A(\ii P_{+}-P_{-})+\smallO(1/k^2),
\label{eq.Masymp1}
\end{equation}
if $\alpha\not=\infty$, and 
\begin{equation}
M_\infty(\lambda)=k(\ii P_{+}+P_{-})+\smallO(1), 
\label{eq.Masymp2}
\end{equation}
if $\alpha=\infty$. 
\end{proposition}
If $\alpha=0$, the expression \eqref{eq.Masymp1}  simplifies to 
\begin{equation}
M_0(\lambda)=\dfrac{1}{k}(\ii P_{+}-P_{-})+\smallO(1/k^2).
\label{eq.Masymp3}
\end{equation}
\begin{proof}

The $M$-functions for different values of $\alpha$ are related by the following explicit transformation (see \cite[Lemma~8.2]{PS3}):
\[
\begin{aligned}
M_{\alpha}(\lambda)&=(\epsilon A+M_0(\lambda))(I-\epsilon A M_0(\lambda))^{-1}, 
&\text{ if $\alpha\not=\infty$,}
\\
M_{\infty}(\lambda)&=-\epsilon M_{0}(\lambda)^{-1}\epsilon, 
&\text{ if $\alpha=\infty$.}
\end{aligned}
\]
This transformation and some simple algebra reduce the proof to the case $\alpha=0$. 
Thus, it suffices to prove \eqref{eq.Masymp3} as $\abs{k}\to\infty$ with $0\leq\arg k\leq\pi/4$. 
We use formula \eqref{e4}, which in the case $\alpha=0$ simplifies to 
\begin{align}
M_0(\lambda)=-E(0,k)(E'(0,k))^{-1}\epsilon.
\label{e4B}
\end{align}
We recall that $E=\{\myF_{+\ii},\myF_{-1}\}$ and the large $\abs{k}$ asymptotics of the Jost solutions $\myF_{-1}$ and $\myF_{+\ii}$ in the sector $0\leq\arg k\leq\pi/4$ has been computed in Lemma~\ref{lma.cont2}. Substituting the asymptotics of Lemma~\ref{lma.cont2} into \eqref{e4B}, after some straightforward algebra we obtain \eqref{eq.Masymp3}. 
\end{proof}

Let $\bH$ (resp. $\widehat{\bH}$) be the operator \eqref{eq:a2} with a potential $q\in L^1(\bbR_+)$ (resp. $\widehat{q}\in L^1(\bbR_+)$) and boundary parameter $\alpha$ (resp. $\widehat{\alpha}$). We will use hats for various quantities associated with $\widehat{\bH}$, such as  $\widehat A$, $\widehat{M}_{\widehat{\alpha}}(\lambda)$ etc. We assume that the spectral measures of $\bH$ and $\widehat{\bH}$ coincide, i.e. $\widehat{\Sigma}=\Sigma$. 
Our first aim is to show that $\widehat{\alpha}=\alpha$ and $\widehat{M}_{\widehat{\alpha}}(\lambda)=M_\alpha(\lambda)$. The lemma below is a slightly adapted version of \cite[Lemma~5.4]{PS3}. 

\begin{lemma}
If $\widehat{\Sigma}=\Sigma$, then  $\widehat{\alpha}=\alpha$ and
$\widehat{M}_{\widehat{\alpha}}(\lambda)=M_\alpha(\lambda)$
for all $\lambda\in\bbC\setminus\bbR$.
\end{lemma}

\begin{proof}
Using the integral representation \eqref{eq:intres}, from  $\widehat{\Sigma}=\Sigma$ we conclude that  
\begin{equation}
\widehat{M}_{\widehat{\alpha}}(\lambda)-M_\alpha(\lambda)=\text{const}.
\label{eq:const}
\end{equation}
By Proposition~\ref{prop:M_func_asympt}, inspecting the asymptotics at infinity, we see that we have two possibilities: either $\widehat{\alpha}=\alpha=\infty$ or both $\widehat{\alpha}\not=\infty$ and $\alpha\not=\infty$. 

\emph{Case 1:} $\widehat{\alpha}=\alpha=\infty$. In this case, by \eqref{eq.Masymp2} we conclude that $M_\infty=\widehat{M}_{\infty}$ and the proof is complete. 

\emph{Case 2:} $\widehat{\alpha}\not=\infty$ and $\alpha\not=\infty$. 
Let us subtract the asymptotics \eqref{eq.Masymp1} for $\widehat{M}_{\widehat{\alpha}}(\lambda)$ from the same asymptotics for $M_\alpha(\lambda)$ and use \eqref{eq:const}. Inspecting the $\calO(1/k)$ terms, we find that $\abs{\widehat{\alpha}}=\abs{\alpha}$. Using this fact and inspecting the $\calO(1/k^2)$ terms, we find that $\widehat{A}=A$, i.e. $\widehat{\alpha}=\alpha$. Now revisiting the $\calO(1)$ terms, we see that they cancel and therefore $\widehat{M}_{\widehat{\alpha}}(\lambda)-M_\alpha(\lambda)=0$.  
\end{proof}
The remainder of the proof of Theorem~1.1 in \cite{PS3} ($\widehat{M}_\alpha=M_\alpha$ implies $\widehat{Q}=Q$) is based entirely on the analysis of the solutions $\Phi$ and $\Theta$ and the corresponding Volterra integral equations; this analysis does not require the boundedness of $q$.

\end{document}